\documentclass[reqno, 12pt]{amsart}
\pdfoutput=1
\makeatletter
\let\origsection=\section \def\section{\@ifstar{\origsection*}{\mysection}} 
\def\mysection{\@startsection{section}{1}\z@{.7\linespacing\@plus\linespacing}{.5\linespacing}{\normalfont\scshape\centering\S}}
\makeatother        
\usepackage{amsmath,amssymb,amsthm}    
\usepackage{verbatim} %for comments
\usepackage{mathabx}\changenotsign   
\usepackage{mathrsfs} 
\usepackage[babel]{microtype}
\usepackage{xcolor}  	
\usepackage[backref]{hyperref}
\hypersetup{
	colorlinks,
    linkcolor={red!60!black},
    citecolor={green!60!black},
    urlcolor={blue!60!black},
}

\usepackage{pgf,tikz}
\usepackage{mathrsfs}
\usetikzlibrary{arrows}

\usepackage{bookmark}

\usepackage[abbrev,msc-links,backrefs]{amsrefs} 
\usepackage{doi}

\renewcommand{\PrintDOI}[1]{\doi{#1}}

\usepackage[T1]{fontenc}
\usepackage{lmodern}

\usepackage[english]{babel}
\numberwithin{equation}{section}

\usepackage{geometry}
\usepackage{enumitem}

\let\polishlcross=\l
\def\l{\ifmmode\ell\else\polishlcross\fi}

\let\emptyset=\varnothing
\let\setminus=\smallsetminus

\makeatletter
\def\moverlay{\mathpalette\mov@rlay}
\def\mov@rlay#1#2{\leavevmode\vtop{   \baselineskip\z@skip \lineskiplimit-\maxdimen
   \ialign{\hfil$\m@th#1##$\hfil\cr#2\crcr}}}
\newcommand{\charfusion}[3][\mathord]{
    #1{\ifx#1\mathop\vphantom{#2}\fi
        \mathpalette\mov@rlay{#2\cr#3}
      }
    \ifx#1\mathop\expandafter\displaylimits\fi}
\makeatother

\DeclareFontFamily{U}  {MnSymbolC}{}
\DeclareSymbolFont{MnSyC}         {U}  {MnSymbolC}{m}{n}
\DeclareFontShape{U}{MnSymbolC}{m}{n}{
    <-6>  MnSymbolC5
   <6-7>  MnSymbolC6
   <7-8>  MnSymbolC7
   <8-9>  MnSymbolC8
   <9-10> MnSymbolC9
  <10-12> MnSymbolC10
  <12->   MnSymbolC12}{}
\DeclareMathSymbol{\powerset}{\mathord}{MnSyC}{180}

\usepackage{tikz}
\usetikzlibrary{decorations.markings}
\usetikzlibrary{calc,positioning,decorations.pathmorphing,decorations.pathreplacing}

\makeatletter
\def\namedlabel#1#2{\begingroup
    #2%
    \def\@currentlabel{#2}%
    \phantomsection\label{#1}\endgroup
}
\makeatother

\theoremstyle{plain}
\newtheorem{theorem}{Theorem}[section]
\newtheorem{fact}[theorem]{Fact}
\newtheorem{prop}[theorem]{Proposition}

\newtheorem{lemma}[theorem]{Lemma}
\newtheorem{claim}[theorem]{Claim}
\newtheorem{conj}[theorem]{Conjecture}

\theoremstyle{definition}

\newtheorem{definition}[theorem]{Definition}

\newtheorem{quest}[theorem]{Question}

\theoremstyle{remark}

\newtheorem{note}[theorem]{Note}

\let\theta=\vartheta
\let\rho=\varrho
\let\phi=\varphi

\newcommand{\cC}{\mathcal{C}}
\newcommand{\cE}{\mathcal{E}}

\newcommand{\cP}{\mathcal{P}}

\newcommand{\tS}{\tilde{S}}

\begin{document}

\title{Embedding Hypertrees into Steiner Triple Systems}

\author[Bradley Elliott]{Bradley Elliott}
\address{Department of Mathematics and Computer Science, 
Emory University, Atlanta, GA 30322, USA}
\email{bradley.elliott@emory.edu}
%\thanks{Thanks 1}

\author[Vojt\v{e}ch R\"{o}dl]{Vojt\v{e}ch R\"{o}dl}
\address{Department of Mathematics and Computer Science, 
Emory University, Atlanta, GA 30322, USA}
\email{rodl@mathcs.emory.edu}
\thanks{The second author was supported by NSF grant DMS 1764385.}

\dedicatory{ To Charlie Colbourn and Alex Rosa on the occasion of their round birthdays}

\begin{abstract}
In this paper we are interested in the following question:  Given an arbitrary Steiner triple system $S$ on $m$ vertices and any 3-uniform hypertree $T$ on $n$ vertices, is it necessary that $S$ contains $T$ as a subgraph provided $m \geq (1+\mu)n$?
We show the answer is positive for a class of hypertrees and conjecture that the answer is always positive.
\end{abstract}


\maketitle

\section{Introduction}
The well-known Tree Packing Conjecture of Gy\'arf\'as and Lehel (\cite{GL}) states that any arbitrary set of trees $T_2,T_3, \cdots, T_{m}$ where tree $T_i$ has order $i$ can be packed into the complete graph $K_m$.
This conjecture remains open despite many partial results since its statement in 1976.
One such result by Bollob\'as (\cite{Bol}) shows that any arbitrary set of trees $T_2, T_3, \cdots, T_s$ can be packing into $K_m$ when $3\leq s < \frac{m}{\sqrt{2}}$.

Inspired by this conjecture, Peter Frankl (personal communication) asked a similar question regarding hypertrees and Steiner triple systems.
A \emph{hypertree} is a connected, simple 3-uniform hypergraph in which every two vertices are joined by a unique path.
Note that any hypertree must have odd order, and in particular any hypertree with size $s$ has order $2s+1$.
A Steiner triple system is a 3-uniform hypergraph in which every pair of vertices is contained in exactly one edge.

\begin{quest}[Frankl]\label{q:Frankl}
What is the largest value of $s$ so that any $s$ hypertrees $T_3, T_5, T_7, \cdots, T_{2s+1}$ can be packed into any Steiner triple system $S$ on $m$ vertices?
\end{quest}

Clearly $s \leq \frac{m-1}{2}$, since no hypertree can have order greater than that of $S$.
A greedy argument easily yields that any hypertree with at most $\frac{m+1}{4}$ edges (and thus $\frac{m+3}{2}$ vertices) can be embedded into any $S$.
Following from this, Frankl showed using a method similar to \cite{Bol} that if $s= \frac{m+1}{4}$, then any $s$ hypertrees $T_3, T_5, T_7, \cdots, T_{2s+1}$ can be packed into any $S$ on $m$ vertices.
It is however less clear how to embed larger trees.

This paper then is motivated by the following question:
\begin{quest}\label{q:embed}
Given a Steiner triple system $S$ on $m$ vertices and a hypertree $T$ on $2s+1$ vertices with $2s+1 < m$ vertices, can one find $T$ in $S$ as a subhypergraph?
\end{quest}

For convenience let $n=2s+1$ be the order of $T$.
One can find examples of pairs $T$ and $S$ showing that if $n$ equals $m$ then the answer is negative.
For example, for $s\geq 3$ let $T$ be the hypertree with edges $\{u, v_i, w_i \}$ for $1 \leq i \leq s-1$ and also the edge $\{w_1, x, y \}$.
One can easily observe that $T$ is not in any Steiner triple system $S$ on $2s+1$ vertices.
In fact even assuming only that $n < m$ may likely not be sufficient to guarantee the embedding of any hypertree on $n$ vertices into any Steiner triple system with $m$ vertices.
Consequently we start with the following more modest conjecture.
\begin{conj}\label{conjec}
Given $\mu > 0$, there exists $n_0 = n_0(\mu)$ such that if $n > n_0$, $T$ is any hypertree on $n$ vertices, and $S$ is any Steiner triple system on $m \geq n(1+\mu)$ vertices, then $T$ is a subhypergraph of $S$.
\end{conj}

Note that if hypertree T is replaced by a matching the analogous result is true -- in other words, every Steiner triple system contains an almost perfect matching.
Various generalizations of this fact are known and we mention some in Section 5.

Unfortunately we are unable to resolve even this Conjecture~\ref{conjec} and will address a specific case only.  Here we consider a special class of trees.

\begin{definition}\label{def:subdiv}
A \emph{subdivision tree} $T$ is a hypertree in which each edge contains a vertex of degree one.
\end{definition}

Equivalently, $T$ can be obtained from a graph tree $T'$ by subdividing each edge $\{x,y\}$ of $T'$ by a vertex $z_{xy}$ and setting
$$V(T) = V(T') \cup \{z_{xy}, \{x,y\} \in E(T')\}$$
$$E(T) = \{ \{ x,y,z_{xy}\}, \{x,y\} \in E(T') \}.$$

We say that a hypertree $T$ has bounded degree $d$ if no vertex of $T$ has degree greater than $d$.

\begin{theorem}[Main Theorem]\label{thm:main}
For any $d\in \mathbb{Z}^+$ and any $\mu > 0$ such that $\frac{1}{d} \gg \mu$, there exists $n_0=n_0(d,\mu)$ such that for all odd $n>n_0$, any subdivistion tree $T$ on $n$ vertices with bounded degree $d$ is a subhypergraph of any Steiner triple system $S$ on $m\geq(1+\mu)n$ vertices.
\end{theorem}

We will divide the proof into four parts. 
First, we will decompose the hypertree $T$ into smaller subhypertrees by removing some edges in the shape of stars and some isolated vertices.
The vast majority of $V(T)$ will be contained in the subhypertrees after the decomposition.
We will keep track of which stars we remove while decomposing $T$ so that we can restore them later.
Second, we will show that given a set of at most $d$ vertices in the Steiner triple system $S$, we can find many stars in $S$ that all contain the vertices but that are otherwise pairwise disjoint.
These stars in $S$ are the candidates for where to eventually embed the stars we removed from $T$ in part 1.
Third, we will fix a subset of the vertices of $S$, called the reservoir.
The reservoir is where the isolated vertices from part one will eventually be embedded.
Lastly, we will embed the subhypertrees from part one into the Steiner triple system $S$, though we will avoid using the reservoir.
Then we will use the reservoir to embed the isolated vertices and stars removed in part one.

The constants used in the proof of Theorem~\ref{thm:main} follow the following hierarchy.
\begin{equation}\label{eq:hier}
1 > \frac{1}{d} \gg \mu \gg \epsilon \gg \frac{1}{k} \gg \frac{1}{t} > \frac{1}{k3^k} \gg \frac{1}{l} \gg \frac{1}{n} > \frac{1}{m},
\end{equation}
where $d$, $\mu$, $n$, and $m$ are as stated in the theorem and the others are defined when needed.
The reader may think about the constants $d$, $\mu$, $\epsilon$, $k$, and $t$ as being fixed while $l$, $n$, and $m$ are tending together to infinity.

\section{Decomposing the Hypertree}

We will decompose $T$ into a set $\cP$ of subhypertrees because the smaller hypertrees will be simpler to embed into $S$.
In the proof of Theorem~\ref{thm:main} we will describe how the embedded subhypertrees can be reassembled to form a copy of $T$ in $S$.

We will need the following definition throughout this paper.

\begin{definition}\label{def:star}
A \emph{star} $S$ in a 3-uniform hypergraph $G$ is a set of edges $\{v_i,w_i,u\}\in E(G)$, $1\leq i\leq c = deg_G(u)$.
All vertices $v_i, w_i, u$, $1\leq i \leq c$ must be distinct, so that any two edges intersect precisely at $u$, which we call the \emph{center} of the star.
\end{definition}

The following lemma describes the result of the decomposition process.

\begin{lemma}\label{lemma:sawing}
Let $T$ be any subdivison tree on $n$ vertices with bounded degree $d\ll n$, and let $k$ be any integer with $d \ll k \ll n$.
Then there exists a system $\cE$ of $e$ stars $E_j = \{\{v_{j,i},w_{j,i},u_{j}\}, 1\leq i \leq deg(u_j)\} \subset E(T)$, for $j=1,\ldots,e$, such that by removing all of the edges of the stars from $T$, $T$ is decomposed into 
\begin{itemize}
\item a set $I$ of isolated vertices, and
\item a set $\cP$ of $l$ subhypertrees
\end{itemize}
with the following properties.

\begin{enumerate}
\item\label{sawing2} $k \geq \vert V(P) \vert \text{, for any } P \in \cP$.
\item\label{sawing1} $\left( \frac{2d^2}{k}\right) n \geq \vert I \vert$.
\item\label{sawing3} $\vert I \vert \geq \vert \cP \vert =l$
\item\label{sawing4} $ l \geq \vert \cE \vert = e $
\item\label{sawing5} $ l \geq \frac{n}{k + 3}$.
\item\label{sawing6} $I = \bigcup_{j=1}^e \{w_{j,1}, w_{j,2}, \ldots, w_{j,deg(u_j)}, u_j\}$, and $v_{j,i} \not \in I$ for all $i,j$.
\end{enumerate}
\end{lemma}

\begin{note}
Some of the subhypertrees in $\cP$ may contain just a single vertex, but for technical reasons they will still be considered as elements of $\cP$ and not of $I$.
\end{note}

Before we prove Lemma~\ref{lemma:sawing}, we introduce some terminology.
We fix some vertex of degree at least 2 to be the root of $T$.
We say two vertices are \emph{adjacent} if they belong to the same edge of $T$.
Define a \emph{leaf} on $T$ to be any degree-one vertex that is adjacent to another degree-one vertex.
Borrowing the terminology of a family tree, for a vertex $v$ of $T$, we say that the \emph{father} of $v$ is the neighbor of $v$ that lies on the path from $v$ to the root.
Likewise, we say that a \emph{son} of $v$ is any neighbor of $v$ whose path to the root passes through $v$.
Note that the root has no father and leaf vertices have no sons.
Define a \emph{branch} of $T$ to be a sequence $\{ b_h\}_{h=1}^L$ of vertices in $T$ where $b_1$ is the root, $b_{h+1}$ is a son of $b_h$, and $b_L$ is a leaf.
We say that the \emph{progeny} of a vertex $v$ is the set of all vertices whose paths to the root must pass through $v$.
That is, the progeny of $v$ is the set of all of $v$'s sons, their sons, their sons, etc.
We say $v$ is included in the set of its own progeny.
Lastly, we say that a vertex is \emph{celibate} if it is the only degree-one vertex in its edge.
If an edge has two degree-one vertices, we choose exactly one of them to call celibate.
In this way, every edge has exactly one celibate vertex.
During the decomposition, we will use the distinction between celibate and non-celibate vertices to decide whether a single vertex should be added to $\cP$ (as a subhypertree) or to $I$ (as an isolated vertex).

\begin{proof}[Proof of Lemma~\ref{lemma:sawing}]\label{proof:sawing}
Choose any vertex with degree at least two to be the root of $T$ and decide which vertices to call celibate.
Create empty sets $I$, $\cP$, and $\cE$, which will be used to store isolated vertices, subhypertrees, and stars (respectively) as $T$ is decomposed.

We assign a proper coloring to $V(T)$ in the following way.
Color all celibate vertices blue.
Color the root red.
For every remaining uncolored vertex, color it red if its father is blue, and blue if its father is red.
In this coloring, every edge has exactly one red and two blue vertices.
A necessary (but not sufficient) condition for a star to belong to $\cE$ will be that its center is red.
This ensures that the centers of two stars are never adjacent vertices, which will be important as we reassemble $T$ in $S$.

To construct $\cE$, and with it $\cP$ and $I$, repeat the following "sawing" procedure, each iteration of which will remove one or more stars,  subhypertrees, and isolated vertices from $T$.
What remains of $T$ at the beginning of the $j^{th}$ iteration is called $H_j$, where $H_1 = T$.
To simplify the notation throughout the proof, we will drop the index $j$ and will write $v_i = v_{j,i}$, $w_i = w_{j,i}$, and $u = u_j$, as well as $b_h = b_{j,h}$, whenever it is clear from the context that $j$ is fixed.

\begin{enumerate}
\item[(a)] Let $\{b_{h}\}$ be a branch of $H_j$, where $b_{h+1}$ is the son of $b_{h}$ that has the most vertices in its progeny.
Let $b_{x}$ be the last vertex in this sequence that has more than $k$ vertices in its progeny.
If $b_x$ is red, let $u = b_{x}$.
If $b_x$ is blue, let $u = b_{x+1}$.
In either case, $u$ is red and has at least $\frac{k-d}{d-1}$ vertices in its progeny.
We will "saw" around the vertex $u$.

\item[(b)] Let $E_j = \{\{v_{i},w_{i},u\}$, $1 \leq i \leq deg(u)\}$ be the star centered at $u$.
Label the vertices adjacent to $u$ such that $w_{i}$ is the celibate vertex in each edge and $v_{1}$ is the father of $u$.
Figure~\ref{beforesaw} shows how all of the vertices around $u$ should be labeled.
Add $E_j$ to $\cE$ and let $H_{j}' = H_j \setminus E_j$.
Figure~\ref{duringsaw1} shows as dotted triangles which edges are removed to form $H_{j}'$.

\item[(c)] Removal of $E_j$ from $H_j$ results in some vertices and subhypergraphs in $H_{j}'$ not being connected to the root, as shown in Figure~\ref{duringsaw2}.
Specifically, $u$ is now isolated, as are the celibate vertices $w_{i}$, $1 \leq i \leq deg(u)$ (because all celibate vertices of $T$ are originally contained in just one edge).
Add these vertices $u$ and $w_{i}$ to $I$.

\item[(d)] For $i\geq 2$, the vertices $v_{i}$ are not connected to the root in $H_j'$.
(Note that $v_{1}$ is the father of $u_j$, so there is still a path from $v_{1}$ to the root.)
Let $P_{j,i}$ be the connected component containing $v_{i}$ for $i \geq 2$, and add each $P_{j,i}$ (even if it is just a single vertex) to $\cP$.
\item[(e)] Define 
$$H_{j+1} = H_j' \setminus \{u\} \setminus \{w_{i}\}_{i=1}^{deg(u)} \setminus \{P_{j,i}\}_{i=2}^{deg(u)}.$$
That is, $H_{j+1}$ is the connected component of $H_j'$ that contains the root.
Figure~\ref{aftersaw} shows that after the isolated vertices and disconnected subhypergraphs are removed from $H_{j}'$, we are left with a smaller hypertree still containing $v_{1}$.
This smaller hypertree is $H_{j+1}$.
\end{enumerate}

There are two cases that can cause this procedure to end.
First, at some point the root could be the only vertex with more than $k$ vertices in its progeny, so the root becomes $u$.
By "sawing" around $u$ and removing all of the isolated vertices and components not connected to the root, we remove every vertex and edge.
This completes the decomposition.
In this case, notice that the number of complete iterations performed of the sawing procedure is $e = \vert \cE \vert$, since exactly one star is added to $\cE$ with each iteration.
Second, at some point there could be no vertex with more than $k$ vertices in its progeny.
If this occurs, add the entire remaining tree $H_j$ to $\cP$.
In this case, notice again that the number of complete iterations performed of the sawing procedure is $e = \vert \cE \vert$.
This implies the following fact.

\begin{fact}\label{prop1}
Either $H_{e+1}$ does not exist, or $H_{e+1}$ is a hypertree with at most $k$ vertices and is a member of $\cP$.
\end{fact}

\begin{figure}[h]
    	\begin{minipage}{0.46\textwidth}
    		\centering
\definecolor{zzttqq}{rgb}{0,0,0} %{0.6,0.2,0.}
\definecolor{qqqqff}{rgb}{0,0,0} %{0.,0.,1.}
\begin{tikzpicture}[line cap=round,line join=round,>=triangle 45,x=1.0cm,y=0.7cm]
%\clip(-7.23567731300529,-6.720316430435862) rectangle (19.44456628798421,6.383697019607706);
\fill[line width=2.pt,color=zzttqq,fill=zzttqq,fill opacity=0.10000000149011612] (6.,1.) -- (4.2,-1.) -- (4.8,-1.) -- cycle;
\fill[line width=2.pt,color=zzttqq,fill=zzttqq,fill opacity=0.10000000149011612] (6.,1.) -- (5.7,-1.) -- (6.3,-1.) -- cycle;
\fill[line width=2.pt,color=zzttqq,fill=zzttqq,fill opacity=0.10000000149011612] (6.,1.) -- (7.2,-1.) -- (7.8,-1.) -- cycle;
\fill[line width=2.pt,color=zzttqq,fill=zzttqq,fill opacity=0.10000000149011612] (6.,1.) -- (5.4,1.) -- (4.1,3.) -- cycle;
\fill[line width=2.pt,color=zzttqq,fill=zzttqq,fill opacity=0.10000000149011612] (4.1,3.) -- (4.3641429644735465,1.0254659940610364) -- (3.806023398089517,1.0254659940610364) -- cycle;
\fill[line width=2.pt,color=zzttqq,fill=zzttqq,fill opacity=0.10000000149011612] (4.1,3.) -- (2.832891846445569,1.0254659940610364) -- (2.3168531886162764,1.0216671244288105) -- cycle;
\fill[line width=2.pt,color=zzttqq,fill=zzttqq,fill opacity=0.10000000149011612] (4.1,3.) -- (3.509291958589649,3.0008696189206994) -- (3.8,5.) -- cycle;
\fill[line width=2.pt,color=zzttqq,fill=zzttqq,fill opacity=0.10000000149011612] (4.8,-1.) -- (3.8657942712621005,-2.9859105973373743) -- (4.41898751506418,-2.998203780532976) -- cycle;
\fill[line width=2.pt,color=zzttqq,fill=zzttqq,fill opacity=0.10000000149011612] (4.8,-1.) -- (5.,-3.) -- (5.6114262850375525,-3.010496963728578) -- cycle;
\fill[line width=2.pt,color=zzttqq,fill=zzttqq,fill opacity=0.10000000149011612] (7.8,-1.) -- (7.,-3.) -- (7.60292196272504,-2.998203780532976) -- cycle;
\fill[line width=2.pt,color=zzttqq,fill=zzttqq,fill opacity=0.10000000149011612] (7.8,-1.) -- (8.15611520652712,-3.010496963728578) -- (8.758481183111607,-3.010496963728578) -- cycle;
\draw [line width=2.pt,color=zzttqq] (6.,1.)-- (4.2,-1.);
\draw [line width=2.pt,color=zzttqq] (4.2,-1.)-- (4.8,-1.);
\draw [line width=2.pt,color=zzttqq] (4.8,-1.)-- (6.,1.);
\draw [line width=2.pt,color=zzttqq] (6.,1.)-- (5.7,-1.);
\draw [line width=2.pt,color=zzttqq] (5.7,-1.)-- (6.3,-1.);
\draw [line width=2.pt,color=zzttqq] (6.3,-1.)-- (6.,1.);
\draw [line width=2.pt,color=zzttqq] (6.,1.)-- (7.2,-1.);
\draw [line width=2.pt,color=zzttqq] (7.2,-1.)-- (7.8,-1.);
\draw [line width=2.pt,color=zzttqq] (7.8,-1.)-- (6.,1.);
\draw [line width=2.pt,color=zzttqq] (6.,1.)-- (5.4,1.);
\draw [line width=2.pt,color=zzttqq] (5.4,1.)-- (4.1,3.);
\draw [line width=2.pt,color=zzttqq] (4.1,3.)-- (6.,1.);
\draw [line width=2.pt,color=zzttqq] (4.1,3.)-- (4.3641429644735465,1.0254659940610364);
\draw [line width=2.pt,color=zzttqq] (4.3641429644735465,1.0254659940610364)-- (3.806023398089517,1.0254659940610364);
\draw [line width=2.pt,color=zzttqq] (3.806023398089517,1.0254659940610364)-- (4.1,3.);
\draw [line width=2.pt,color=zzttqq] (4.1,3.)-- (2.832891846445569,1.0254659940610364);
\draw [line width=2.pt,color=zzttqq] (2.832891846445569,1.0254659940610364)-- (2.3168531886162764,1.0216671244288105);
\draw [line width=2.pt,color=zzttqq] (2.3168531886162764,1.0216671244288105)-- (4.1,3.);
\draw [line width=2.pt,color=zzttqq] (4.1,3.)-- (3.509291958589649,3.0008696189206994);
\draw [line width=2.pt,color=zzttqq] (3.509291958589649,3.0008696189206994)-- (3.8,5.);
\draw [line width=2.pt,color=zzttqq] (3.8,5.)-- (4.1,3.);
\draw [line width=2.pt,color=zzttqq] (4.8,-1.)-- (3.8657942712621005,-2.9859105973373743);
\draw [line width=2.pt,color=zzttqq] (3.8657942712621005,-2.9859105973373743)-- (4.41898751506418,-2.998203780532976);
\draw [line width=2.pt,color=zzttqq] (4.41898751506418,-2.998203780532976)-- (4.8,-1.);
\draw [line width=2.pt,color=zzttqq] (4.8,-1.)-- (5.,-3.);
\draw [line width=2.pt,color=zzttqq] (5.,-3.)-- (5.6114262850375525,-3.010496963728578);
\draw [line width=2.pt,color=zzttqq] (5.6114262850375525,-3.010496963728578)-- (4.8,-1.);
\draw [line width=2.pt,color=zzttqq] (7.8,-1.)-- (7.,-3.);
\draw [line width=2.pt,color=zzttqq] (7.,-3.)-- (7.60292196272504,-2.998203780532976);
\draw [line width=2.pt,color=zzttqq] (7.60292196272504,-2.998203780532976)-- (7.8,-1.);
\draw [line width=2.pt,color=zzttqq] (7.8,-1.)-- (8.15611520652712,-3.010496963728578);
\draw [line width=2.pt,color=zzttqq] (8.15611520652712,-3.010496963728578)-- (8.758481183111607,-3.010496963728578);
\draw [line width=2.pt,color=zzttqq] (8.758481183111607,-3.010496963728578)-- (7.8,-1.);
\begin{scriptsize}
\draw [fill=qqqqff] (6.,1.) circle (2.5pt);
\draw[color=qqqqff] (6.306823073590519,1.062826693367364) node {u};
\draw [fill=qqqqff] (4.2,-1.) circle (2.5pt);
\draw[color=qqqqff] (3.7011987775393953,-0.9356618443805933) node {$w_{2}$};
\draw [fill=qqqqff] (4.8,-1.) circle (2.5pt);
\draw[color=qqqqff] (5.117848880246802,-0.9019320800304167) node {$v_{2}$};
\draw [fill=qqqqff] (5.7,-1.) circle (2.5pt);
\draw[color=qqqqff] (5.573200698974183,-1.188635077006917) node {$w_{3}$};
\draw [fill=qqqqff] (6.3,-1.) circle (2.5pt);
\draw[color=qqqqff] (6.332120396853152,-1.188635077006917) node {$v_{3}$};
\draw [fill=qqqqff] (7.2,-1.) circle (2.5pt);
\draw[color=qqqqff] (7.057310330381943,-1.188635077006917) node {$w_{4}$};
\draw [fill=qqqqff] (7.8,-1.) circle (2.5pt);
\draw[color=qqqqff] (8.153527671762674,-0.9356618443805933) node {$v_{4}$};
\draw [fill=qqqqff] (5.4,1.) circle (2.5pt);
\draw[color=qqqqff] (4.94920005849592,1.05439425227982) node {$w_{1}$};
\draw [fill=qqqqff] (4.1,3.) circle (2.5pt);
\draw[color=qqqqff] (4.392658946718011,3.162504524165851) node {$v_{1}$};
\draw [fill=qqqqff] (3.8,5.) circle (2.5pt);
\draw[color=qqqqff] (4.4179562699806425,5.312777001489602) node {toward root};
\end{scriptsize}
\end{tikzpicture}
\caption{$H_j$ with vertices of star labeled\label{beforesaw}}
    	\end{minipage}  
    \begin{minipage}{0.46\textwidth}
    	\centering
	\definecolor{zzttqq}{rgb}{0,0,0} %{0.6,0.2,0.}
\definecolor{qqqqff}{rgb}{0,0,0} %{0.,0.,1.}
\begin{tikzpicture}[line cap=round,line join=round,>=triangle 45,x=1.0cm,y=0.7cm]
%\clip(-7.23567731300529,-6.720316430435862) rectangle (19.44456628798421,6.383697019607706);
\fill[line width=2.pt,dash pattern=on 4pt off 4pt,color=zzttqq,fill opacity=0] (6.,1.) -- (4.2,-1.) -- (4.8,-1.) -- cycle;
\fill[line width=2.pt,dash pattern=on 4pt off 4pt,color=zzttqq,fill opacity=0] (6.,1.) -- (5.7,-1.) -- (6.3,-1.) -- cycle;
\fill[line width=2.pt,dash pattern=on 4pt off 4pt,color=zzttqq,fill opacity=0] (6.,1.) -- (7.2,-1.) -- (7.8,-1.) -- cycle;
\fill[line width=2.pt,dash pattern=on 4pt off 4pt,color=zzttqq,fill opacity=0] (6.,1.) -- (5.4,1.) -- (4.1,3.) -- cycle;
\fill[line width=2.pt,color=zzttqq,fill=zzttqq,fill opacity=0.10000000149011612] (4.1,3.) -- (4.3641429644735465,1.0254659940610364) -- (3.806023398089517,1.0254659940610364) -- cycle;
\fill[line width=2.pt,color=zzttqq,fill=zzttqq,fill opacity=0.10000000149011612] (4.1,3.) -- (2.832891846445569,1.0254659940610364) -- (2.3168531886162764,1.0216671244288105) -- cycle;
\fill[line width=2.pt,color=zzttqq,fill=zzttqq,fill opacity=0.10000000149011612] (4.1,3.) -- (3.509291958589649,3.0008696189206994) -- (3.8,5.) -- cycle;
\fill[line width=2.pt,color=zzttqq,fill=zzttqq,fill opacity=0.10000000149011612] (4.8,-1.) -- (3.8657942712621005,-2.9859105973373743) -- (4.41898751506418,-2.998203780532976) -- cycle;
\fill[line width=2.pt,color=zzttqq,fill=zzttqq,fill opacity=0.10000000149011612] (4.8,-1.) -- (5.,-3.) -- (5.6114262850375525,-3.010496963728578) -- cycle;
\fill[line width=2.pt,color=zzttqq,fill=zzttqq,fill opacity=0.10000000149011612] (7.8,-1.) -- (7.,-3.) -- (7.60292196272504,-2.998203780532976) -- cycle;
\fill[line width=2.pt,color=zzttqq,fill=zzttqq,fill opacity=0.10000000149011612] (7.8,-1.) -- (8.15611520652712,-3.010496963728578) -- (8.758481183111607,-3.010496963728578) -- cycle;
\draw [line width=2.pt,dash pattern=on 4pt off 4pt,color=zzttqq] (6.,1.)-- (4.2,-1.);
\draw [line width=2.pt,dash pattern=on 4pt off 4pt,color=zzttqq] (4.2,-1.)-- (4.8,-1.);
\draw [line width=2.pt,dash pattern=on 4pt off 4pt,color=zzttqq] (4.8,-1.)-- (6.,1.);
\draw [line width=2.pt,dash pattern=on 4pt off 4pt,color=zzttqq] (6.,1.)-- (5.7,-1.);
\draw [line width=2.pt,dash pattern=on 4pt off 4pt,color=zzttqq] (5.7,-1.)-- (6.3,-1.);
\draw [line width=2.pt,dash pattern=on 4pt off 4pt,color=zzttqq] (6.3,-1.)-- (6.,1.);
\draw [line width=2.pt,dash pattern=on 4pt off 4pt,color=zzttqq] (6.,1.)-- (7.2,-1.);
\draw [line width=2.pt,dash pattern=on 4pt off 4pt,color=zzttqq] (7.2,-1.)-- (7.8,-1.);
\draw [line width=2.pt,dash pattern=on 4pt off 4pt,color=zzttqq] (7.8,-1.)-- (6.,1.);
\draw [line width=2.pt,dash pattern=on 4pt off 4pt,color=zzttqq] (6.,1.)-- (5.4,1.);
\draw [line width=2.pt,dash pattern=on 4pt off 4pt,color=zzttqq] (5.4,1.)-- (4.1,3.);
\draw [line width=2.pt,dash pattern=on 4pt off 4pt,color=zzttqq] (4.1,3.)-- (6.,1.);
\draw [line width=2.pt,color=zzttqq] (4.1,3.)-- (4.3641429644735465,1.0254659940610364);
\draw [line width=2.pt,color=zzttqq] (4.3641429644735465,1.0254659940610364)-- (3.806023398089517,1.0254659940610364);
\draw [line width=2.pt,color=zzttqq] (3.806023398089517,1.0254659940610364)-- (4.1,3.);
\draw [line width=2.pt,color=zzttqq] (4.1,3.)-- (2.832891846445569,1.0254659940610364);
\draw [line width=2.pt,color=zzttqq] (2.832891846445569,1.0254659940610364)-- (2.3168531886162764,1.0216671244288105);
\draw [line width=2.pt,color=zzttqq] (2.3168531886162764,1.0216671244288105)-- (4.1,3.);
\draw [line width=2.pt,color=zzttqq] (4.1,3.)-- (3.509291958589649,3.0008696189206994);
\draw [line width=2.pt,color=zzttqq] (3.509291958589649,3.0008696189206994)-- (3.8,5.);
\draw [line width=2.pt,color=zzttqq] (3.8,5.)-- (4.1,3.);
\draw [line width=2.pt,color=zzttqq] (4.8,-1.)-- (3.8657942712621005,-2.9859105973373743);
\draw [line width=2.pt,color=zzttqq] (3.8657942712621005,-2.9859105973373743)-- (4.41898751506418,-2.998203780532976);
\draw [line width=2.pt,color=zzttqq] (4.41898751506418,-2.998203780532976)-- (4.8,-1.);
\draw [line width=2.pt,color=zzttqq] (4.8,-1.)-- (5.,-3.);
\draw [line width=2.pt,color=zzttqq] (5.,-3.)-- (5.6114262850375525,-3.010496963728578);
\draw [line width=2.pt,color=zzttqq] (5.6114262850375525,-3.010496963728578)-- (4.8,-1.);
\draw [line width=2.pt,color=zzttqq] (7.8,-1.)-- (7.,-3.);
\draw [line width=2.pt,color=zzttqq] (7.,-3.)-- (7.60292196272504,-2.998203780532976);
\draw [line width=2.pt,color=zzttqq] (7.60292196272504,-2.998203780532976)-- (7.8,-1.);
\draw [line width=2.pt,color=zzttqq] (7.8,-1.)-- (8.15611520652712,-3.010496963728578);
\draw [line width=2.pt,color=zzttqq] (8.15611520652712,-3.010496963728578)-- (8.758481183111607,-3.010496963728578);
\draw [line width=2.pt,color=zzttqq] (8.758481183111607,-3.010496963728578)-- (7.8,-1.);
\begin{scriptsize}
\draw [fill=qqqqff] (6.,1.) circle (2.5pt);
\draw[color=qqqqff] (6.306823073590519,1.062826693367364) node {u};
\draw [fill=qqqqff] (4.2,-1.) circle (2.5pt);
\draw[color=qqqqff] (3.7011987775393953,-0.9356618443805933) node {$w_{2}$};
\draw [fill=qqqqff] (4.8,-1.) circle (2.5pt);
\draw[color=qqqqff] (5.117848880246802,-0.9019320800304167) node {$v_{2}$};
\draw [fill=qqqqff] (5.7,-1.) circle (2.5pt);
\draw[color=qqqqff] (5.573200698974183,-1.188635077006917) node {$w_{3}$};
\draw [fill=qqqqff] (6.3,-1.) circle (2.5pt);
\draw[color=qqqqff] (6.332120396853152,-1.188635077006917) node {$v_{3}$};
\draw [fill=qqqqff] (7.2,-1.) circle (2.5pt);
\draw[color=qqqqff] (7.057310330381943,-1.188635077006917) node {$w_{4}$};
\draw [fill=qqqqff] (7.8,-1.) circle (2.5pt);
\draw[color=qqqqff] (8.153527671762674,-0.9356618443805933) node {$v_{4}$};
\draw [fill=qqqqff] (5.4,1.) circle (2.5pt);
\draw[color=qqqqff] (4.94920005849592,1.05439425227982) node {$w_{1}$};
\draw [fill=qqqqff] (4.1,3.) circle (2.5pt);
\draw[color=qqqqff] (4.392658946718011,3.162504524165851) node {$v_{1}$};
\draw [fill=qqqqff] (3.8,5.) circle (2.5pt);
\draw[color=qqqqff] (4.4179562699806425,5.312777001489602) node {toward root};
\end{scriptsize}
\end{tikzpicture}
    	\caption{$H_j$, where the edges of star $E_J$ are the dashed triangles \label{duringsaw1}}
    	\end{minipage}
    \end{figure}

\begin{figure}[h]
    	\begin{minipage}{0.46\textwidth}
    		\centering
\definecolor{zzttqq}{rgb}{0,0,0} %{0.6,0.2,0.}
\definecolor{qqqqff}{rgb}{0,0,0} %{0.,0.,1.}
\begin{tikzpicture}[line cap=round,line join=round,>=triangle 45,x=1.0cm,y=0.7cm]
%\clip(-7.23567731300529,-6.720316430435862) rectangle (19.44456628798421,6.383697019607706);
\fill[line width=2.pt,color=zzttqq,fill=zzttqq,fill opacity=0.10000000149011612] (4.1,3.) -- (4.3641429644735465,1.0254659940610364) -- (3.806023398089517,1.0254659940610364) -- cycle;
\fill[line width=2.pt,color=zzttqq,fill=zzttqq,fill opacity=0.10000000149011612] (4.1,3.) -- (2.832891846445569,1.0254659940610364) -- (2.3168531886162764,1.0216671244288105) -- cycle;
\fill[line width=2.pt,color=zzttqq,fill=zzttqq,fill opacity=0.10000000149011612] (4.1,3.) -- (3.509291958589649,3.0008696189206994) -- (3.8,5.) -- cycle;
\fill[line width=2.pt,color=zzttqq,fill=zzttqq,fill opacity=0.10000000149011612] (4.8,-1.) -- (3.8657942712621005,-2.9859105973373743) -- (4.41898751506418,-2.998203780532976) -- cycle;
\fill[line width=2.pt,color=zzttqq,fill=zzttqq,fill opacity=0.10000000149011612] (4.8,-1.) -- (5.,-3.) -- (5.6114262850375525,-3.010496963728578) -- cycle;
\fill[line width=2.pt,color=zzttqq,fill=zzttqq,fill opacity=0.10000000149011612] (7.8,-1.) -- (7.,-3.) -- (7.60292196272504,-2.998203780532976) -- cycle;
\fill[line width=2.pt,color=zzttqq,fill=zzttqq,fill opacity=0.10000000149011612] (7.8,-1.) -- (8.15611520652712,-3.010496963728578) -- (8.758481183111607,-3.010496963728578) -- cycle;
\draw [line width=2.pt,color=zzttqq] (4.1,3.)-- (4.3641429644735465,1.0254659940610364);
\draw [line width=2.pt,color=zzttqq] (4.3641429644735465,1.0254659940610364)-- (3.806023398089517,1.0254659940610364);
\draw [line width=2.pt,color=zzttqq] (3.806023398089517,1.0254659940610364)-- (4.1,3.);
\draw [line width=2.pt,color=zzttqq] (4.1,3.)-- (2.832891846445569,1.0254659940610364);
\draw [line width=2.pt,color=zzttqq] (2.832891846445569,1.0254659940610364)-- (2.3168531886162764,1.0216671244288105);
\draw [line width=2.pt,color=zzttqq] (2.3168531886162764,1.0216671244288105)-- (4.1,3.);
\draw [line width=2.pt,color=zzttqq] (4.1,3.)-- (3.509291958589649,3.0008696189206994);
\draw [line width=2.pt,color=zzttqq] (3.509291958589649,3.0008696189206994)-- (3.8,5.);
\draw [line width=2.pt,color=zzttqq] (3.8,5.)-- (4.1,3.);
\draw [line width=2.pt,color=zzttqq] (4.8,-1.)-- (3.8657942712621005,-2.9859105973373743);
\draw [line width=2.pt,color=zzttqq] (3.8657942712621005,-2.9859105973373743)-- (4.41898751506418,-2.998203780532976);
\draw [line width=2.pt,color=zzttqq] (4.41898751506418,-2.998203780532976)-- (4.8,-1.);
\draw [line width=2.pt,color=zzttqq] (4.8,-1.)-- (5.,-3.);
\draw [line width=2.pt,color=zzttqq] (5.,-3.)-- (5.6114262850375525,-3.010496963728578);
\draw [line width=2.pt,color=zzttqq] (5.6114262850375525,-3.010496963728578)-- (4.8,-1.);
\draw [line width=2.pt,color=zzttqq] (7.8,-1.)-- (7.,-3.);
\draw [line width=2.pt,color=zzttqq] (7.,-3.)-- (7.60292196272504,-2.998203780532976);
\draw [line width=2.pt,color=zzttqq] (7.60292196272504,-2.998203780532976)-- (7.8,-1.);
\draw [line width=2.pt,color=zzttqq] (7.8,-1.)-- (8.15611520652712,-3.010496963728578);
\draw [line width=2.pt,color=zzttqq] (8.15611520652712,-3.010496963728578)-- (8.758481183111607,-3.010496963728578);
\draw [line width=2.pt,color=zzttqq] (8.758481183111607,-3.010496963728578)-- (7.8,-1.);
\begin{scriptsize}
\draw [color=qqqqff] (6.,1.)-- ++(-2.5pt,-2.5pt) -- ++(5.0pt,5.0pt) ++(-5.0pt,0) -- ++(5.0pt,-5.0pt);
\draw[color=qqqqff] (6.306823073590519,1.062826693367364) node {u};
\draw [color=qqqqff] (4.2,-1.)-- ++(-2.5pt,-2.5pt) -- ++(5.0pt,5.0pt) ++(-5.0pt,0) -- ++(5.0pt,-5.0pt);
\draw[color=qqqqff] (3.7011987775393953,-0.9356618443805933) node {$w_{2}$};
\draw [fill=qqqqff] (4.8,-1.) circle (2.5pt);
\draw[color=qqqqff] (5.117848880246802,-0.9019320800304167) node {$v_{2}$};
\draw [color=qqqqff] (5.7,-1.)-- ++(-2.5pt,-2.5pt) -- ++(5.0pt,5.0pt) ++(-5.0pt,0) -- ++(5.0pt,-5.0pt);
\draw[color=qqqqff] (5.674389992024713,-1.25609460570727) node {$w_{3}$};
\draw [fill=qqqqff] (6.3,-1.) circle (2.5pt);
\draw[color=qqqqff] (6.332120396853152,-1.188635077006917) node {$v_{3}$};
\draw [color=qqqqff] (7.2,-1.)-- ++(-2.5pt,-2.5pt) -- ++(5.0pt,5.0pt) ++(-5.0pt,0) -- ++(5.0pt,-5.0pt);
\draw[color=qqqqff] (7.057310330381943,-1.188635077006917) node {$w_{4}$};
\draw [fill=qqqqff] (7.8,-1.) circle (2.5pt);
\draw[color=qqqqff] (8.153527671762674,-0.9356618443805933) node {$v_{4}$};
\draw [color=qqqqff] (5.4,1.)-- ++(-2.5pt,-2.5pt) -- ++(5.0pt,5.0pt) ++(-5.0pt,0) -- ++(5.0pt,-5.0pt);
\draw[color=qqqqff] (4.94920005849592,1.05439425227982) node {$w_{1}$};
\draw [fill=qqqqff] (4.1,3.) circle (2.5pt);
\draw[color=qqqqff] (4.392658946718011,3.162504524165851) node {$v_{1}$};
\draw [fill=qqqqff] (3.8,5.) circle (2.5pt);
\draw[color=qqqqff] (4.4179562699806425,5.312777001489602) node {toward root};
\end{scriptsize}
\end{tikzpicture}
\caption{$H_j'$. The vertices marked with an x will be added to $I$.  The subhypertrees containing $v_2$, $v_3$, and $v_4$ will be added to $\cP$ \label{duringsaw2}}
    	\end{minipage}  
    \begin{minipage}{0.46\textwidth}
    	\centering
	\definecolor{zzttqq}{rgb}{0,0,0} %{0.6,0.2,0.}
\definecolor{qqqqff}{rgb}{0,0,0} %{0.,0.,1.}
\begin{tikzpicture}[line cap=round,line join=round,>=triangle 45,x=1.0cm,y=0.7cm]
%\clip(-7.23567731300529,-6.720316430435862) rectangle (19.44456628798421,6.383697019607706);
\fill[line width=2.pt,color=zzttqq,fill=zzttqq,fill opacity=0.10000000149011612] (4.1,3.) -- (4.3641429644735465,1.0254659940610364) -- (3.806023398089517,1.0254659940610364) -- cycle;
\fill[line width=2.pt,color=zzttqq,fill=zzttqq,fill opacity=0.10000000149011612] (4.1,3.) -- (2.832891846445569,1.0254659940610364) -- (2.3168531886162764,1.0216671244288105) -- cycle;
\fill[line width=2.pt,color=zzttqq,fill=zzttqq,fill opacity=0.10000000149011612] (4.1,3.) -- (3.509291958589649,3.0008696189206994) -- (3.8,5.) -- cycle;
\draw [line width=2.pt,color=zzttqq] (4.1,3.)-- (4.3641429644735465,1.0254659940610364);
\draw [line width=2.pt,color=zzttqq] (4.3641429644735465,1.0254659940610364)-- (3.806023398089517,1.0254659940610364);
\draw [line width=2.pt,color=zzttqq] (3.806023398089517,1.0254659940610364)-- (4.1,3.);
\draw [line width=2.pt,color=zzttqq] (4.1,3.)-- (2.832891846445569,1.0254659940610364);
\draw [line width=2.pt,color=zzttqq] (2.832891846445569,1.0254659940610364)-- (2.3168531886162764,1.0216671244288105);
\draw [line width=2.pt,color=zzttqq] (2.3168531886162764,1.0216671244288105)-- (4.1,3.);
\draw [line width=2.pt,color=zzttqq] (4.1,3.)-- (3.509291958589649,3.0008696189206994);
\draw [line width=2.pt,color=zzttqq] (3.509291958589649,3.0008696189206994)-- (3.8,5.);
\draw [line width=2.pt,color=zzttqq] (3.8,5.)-- (4.1,3.);
\begin{scriptsize}
\draw [fill=qqqqff] (4.1,3.) circle (2.5pt);
\draw[color=qqqqff] (4.402779313850565,3.1690798886590237) node {$v_{1}$};
\draw [fill=qqqqff] (3.8,5.) circle (2.5pt);
\draw[color=qqqqff] (4.446649363730799,5.318712332790555) node {toward root};
\end{scriptsize}
\end{tikzpicture}
    	\caption{$H_{j+1}$, which is the connected component of $H_j'$ containing the root \label{aftersaw}}
    	\end{minipage}
    \end{figure}

To prove~(\ref{sawing2}) of Lemma~\ref{lemma:sawing}, we recall that our choice of $u$ in step (a) for $j \leq e$ implies that all sons of $u$ have at most $k$ vertices in their progeny.
These sons and their progeny make up the vertex sets of the subhypergraphs that we add to $\cP$ in step (d).
This, along with Fact~\ref{prop1}, represents the only two ways that subhypergraphs can be added to $\cP$.
We know then that
$$k \geq \vert V(P) \vert \text{, for any} \in \cP$$
which proves~(\ref{sawing2}) of Lemma~\ref{lemma:sawing}.

To prove~(\ref{sawing1}), we need to show that very few of the vertices of $T$ end up in $I$ (and therefore, most of the vertices of $T$ will be in subhypertrees in $\cP$).
For all $j \leq e$, 
$$\vert V(H_{j}) \setminus V(H_{j+1})\vert > \frac{k-d}{d-1}.$$
This is because in $H_j$, $u$ has more than $\frac{k-d}{d-1}$ vertices in its progeny, and none of the progeny are in $V(H_{j+1})$.
With each iteration of the sawing procedure, at most $d+1$ vertices are added to $I$: one for $u$ and at most $d$ for the celibate neighbors of $u$.
The remaining vertices in $V(H_{j}) \setminus V(H_{j+1})$ must therefore be in an element of $\cP$ (by step (d)), meaning that $\sum_{P \in \cP} \vert V(P) \vert$ increases by at least $\frac{k-d}{d-1} - (d+1)$ with each iteration.

By Fact~\ref{prop1}, if $H_{e+1}$ exists, one subhypertree is added to $\cP$ and no vertices are added to $I$.
This implies that at the end of the decomposition,
$$\frac{\sum_{P \in \cP} \vert V(P) \vert}{\vert I \vert} \geq \frac{\frac{k-d}{d-1} - (d+1)}{d+1} = \frac{k-d}{d^2-1}-1,$$
and since $\vert I \vert + \sum_{P \in \cP} \vert V(P) \vert = \vert V(T) \vert = n$, we have that
$$\vert I \vert \leq \left( \frac{d^2-1}{k-d}\right) n \leq \left( \frac{2d^2}{k}\right) n,$$
which is~(\ref{sawing1}) of Lemma~\ref{lemma:sawing}.

Now we establish~(\ref{sawing3}) of the lemma.
As described in step (d), for every subhypergraph $P \in \cP$, $P$ contains at least one vertex $v_{i}$.
Each of these $v_{i}$ has a corresponding isolated vertex $w_{i}$ that was placed into $I$ in step (c).
Therefore
\begin{equation*}\label{eq1}
\vert I \vert \geq \vert \cP \vert,
\end{equation*}
proving~(\ref{sawing3}).

Proving~(\ref{sawing4}) is similarly clear.
For every iteration of the sawing procedure, one star is added to $\cE$ (in step (b)), and at least
one subhypertree is added to $\cP$ (in step (d)). 
Therefore
$$ \vert \cP \vert \geq \vert \cE \vert.$$.

To find the minimum cardinality of $\cP$ over all hypertrees on $n$ vertices (and show~(\ref{sawing5})), consider how many vertices of $T$ are removed during a single iteration of the sawing procedure.
By step (c), $1+deg(u)$ vertices are put into $I$.
By step (d), $\sum_{i=2}^{deg(u)} \vert V(P_{j,i})\vert$ vertices are put into $\cP$.
These vertices are divided between $deg(u)-1$ subhypertrees.
Therefore the number of vertices removed per subhypertree added to $\cP$ is at most
$$\frac{1 + deg(u) + \sum_{i=2}^{deg(u)} \vert V(P_{j,i})\vert }{deg(u) - 1} \leq \frac{1 + deg(u) + \sum_{i=2}^{deg(u)} k}{deg(u) - 1} \leq k + 3.$$
Summing over all iterations of the sawing procedure, we get that
$$\frac{n}{\vert \cP \vert} \leq k+3,$$
proving~(\ref{sawing5}).

In order to verify~(\ref{sawing6}), recall that by step (c), the only vertices added to $I$ are of the form $w_{j,i}$ and $u_j$, for some $i,j$.
For some vertex $v_{j,i}$ to be in $I$, $v_{j,i}$ would need to be the same vertex as some $u_h$.  ($v_{j,i}$ is not celibate so it cannot be some $w_{h,l}$.)
By the coloring of $V(T)$, $u_h$ is red.
However, $v_{j,i}$ is blue, since it is adjacent to the center (red) vertex $u_j$ of star $E_j$, and in the coloring no two red vertices are adjacent.
Therefore $v_{j,i} \not \in I$.
This completes the proof of~(\ref{sawing6}) and of Lemma~\ref{lemma:sawing}.

\end{proof}

\section{Searching for Disjoint Stars}

\begin{lemma}\label{lemma:stars}
Let $S$ be a Steiner triple system on $m$ vertices.
Then for any set of vertices $v_1, v_2, \ldots, v_c \in V(S)$, $c\leq d$, there are $s(c) = \frac{m}{c^2+1}$ stars $S_1, \ldots, S_{s(c)}$ so that
\begin{enumerate}

\item $S_l = \{ \{v_i,w_i^{(l)},u^{(l)}\}, 1 \leq i \leq c\} \subset E(S)$ is a star centered at $u^{(l)}$, for $l=1,\ldots,s(c)$.
\item The sets $W_l = \{w_1^{(l)},\ldots,w_c^{(l)},u^{(l)}\}$, $l=1,\ldots,s(c)$, are pairwise disjoint subsets of $V(S)$.
\end{enumerate}
\end{lemma}

\begin{proof}\label{proof:stars}
Fix any $v_1, v_2, \ldots, v_c \in V(S)$. 
We prove the lemma by induction, in each step constructing a new star.
As a base case, we construct a set $V^{(1)}\subset V(S)$ of vertices that may serve as the center $u^{(1)}$ of the first star $S_1$.
Let
$$Q = \{ u\in V(S): \exists v_i, v_j \text{ such that } \{u, v_i, v_j\} \in E(S)\}.$$
The vertices of $Q$ cannot be used as $u^{(1)}$.
Set
$$V^{(1)} = V(S) \setminus  \{v_1, v_2, \ldots, v_c\} \setminus Q $$
and observe that
$$\vert V^{(1)} \vert \geq m -c -\binom{c}{2} > m-(c^2+1).$$

Select any vertex in $V^{(1)}$ and call it $u^{(1)}$.
The vertex $u^{(1)}$ and each vertex $v_i$, $1\leq i \leq c$, share an edge with some vertex $w_i^{(1)} \in V(S)$.
Each $w_i^{(1)}$ is distinct from each $v_j$, because otherwise $u^{(1)}$ would be in $Q$.
The edges $\{ v_i, w_i^{(1)}, u^{(1)} \}$, $1 \leq i \leq c$, form a star $S_1$ in $S$.

Now we consider the set
$$Q^{(1)} = \{ u_{ij}^{(1)}\in V(S): \exists v_i, w_j^{(1)}, i\neq j \text{ such that } \{v_i, w_j^{(1)}, u_{ij}^{(1)}\} \in E(S)\}.$$
Note that $\vert Q^{(1)} \vert \leq c(c-1)$.
If we attempted using $u_{ij}^{(1)}\in Q^{(1)}$ as the center of some other star, that star would contain the edge $\{v_i, w_j^{(1)}, u_{ij}^{(1)}\}$, so two different stars would use the vertex $w_j^{(1)}$.
Excluding these vertices $u_{ij}^{(1)}$ from the center of any future star $S_l, l\geq 2$, ensures that 
$$V(S_1) \cap V(S_l) = \{v_1, \ldots, v_c\}$$
and so
$$W_1 \cap W_l = \emptyset.$$
Set
$$V^{(2)} = V^{(1)} \setminus  W_1 \setminus Q^{(1)} .$$
It follows that 
$$\vert V^{(2)} \vert \geq \vert V^{(1)} \vert -(c+1) -c(c-1) > m - 2(c^2+1).$$
Thus we have constructed one star $S_1$ and the set $V^{(2)}$ such that any star $S_l$ whose center $u^{(l)}$ is in $V^{(2)}$ will have $W_l$ disjoint from $W_1$.

Having completed the base case, we assume by induction that the stars $S_1,\ldots, S_{l-1}$ have been constructed.
We assume the set $V^{(l)}$ has the property that for any star $S_l$ whose center is in $V^{(l)}$, $W_l$ will be disjoint from $W_1$, $W_2$, \ldots, $W_{l-1}$.
We also assume inductively that $\vert V^{(l)} \vert > m - l(c^2+1)$.
Choose any vertex $u^{(l)} \in V^{(l)}$ to be the center of star $S^{(l)}$.
For each $i$, $1\leq i \leq c$, let $w_i^{(l)}$ be the unique vertex with $\{v_i, w_i^{(l)}, u^{(l)}\} \in E(S)$.
Observe that by induction,
$$w_i^{(l)} \not\in \bigcup_{k<l} V(S_k).$$
The edges $\{ v_i, w_i^{(l)}, u^{(l)} \}$, $1 \leq i \leq c$, form a star $S_l$ in $S$, and 
$$W_l \cap \left( \bigcup_{k<l} W_k \right) = \emptyset.$$

Let
$$Q^{(l)} = \{ u_{ij}^{(l)}\in V(S): \exists v_i, w_j^{(l)}, i\neq j \text{ such that } \{v_i, w_j^{(l)}, u_{ij}^{(l)}\} \in E(S)\}.$$
and set 
$$V^{(l+1)} = V^{(l)} \setminus  W_l \setminus Q^{(l)} .$$
For the same reason as in the base case, a star centered at any $u^{(l+1)} \in V^{(l+1)}$ will intersect each of $S_1,\ldots, S_{l}$ precisely at $v_1, \ldots, v_c$.
By the inductive hypothesis,
\begin{equation}\label{eq3}
\vert V^{(l+1)}\vert \geq \vert V^{(l)} \vert - (c+1) - c(c-1) > m - (l+1) (c^2+1).
\end{equation}
This completes the induction.

We can continue forming new stars $S_l$ from sets $V^{(l)}$ in this way as long as $V^{(l)}$ has at least one vertex in it to become $u^{(l)}$.
Therefore, by equation~(\ref{eq3}), we can construct at least $\frac{m}{c^2+1}$ stars.
The same process can be conducted for any $c$-tuple, where $c\leq d$, completing the proof.
\end{proof}

\section{Selecting the Reservoir}

Let $\cP = \{ P_1, P_2, \ldots, P_l\}$ be an enumerated collection of hypertrees as formed by the decomposition of $T$ in Lemma~\ref{lemma:sawing}, and $I$ the set of independent vertices formed by the decomposition.
Consider
\begin{equation}\label{Psum}
P = \bigvee_{j=1}^l P_j
\end{equation}
as a forest consisting of (vertex-disjoint) members of $\cP$.
Then
$$V(T) = I \cup V(P).$$
Next we are going to randomly select a set $R \subset V(S)$ and show that its properties allow us to find $T$ in $S$ in such a way that $I \subset R$ and $V(P) \subset V(S) \setminus R$.
We will call the set $R$ the \emph{reservoir}.

Let $a\sim b$ mean that $\lim_{m\to \infty} \frac{a}{b} = 1$.

\begin{lemma}\label{lemma:random}
Let $S = (V,E)$ be a Steiner triple system on $m$ vertices, and let $\epsilon>0$ be small.
Then there exists a subset $R \subset V$ and a hypergraph $\tS = S[V \setminus R]$ induced on the set $V\setminus R$ that have the following properties.
\begin{enumerate}
\item $\vert R \vert \sim \epsilon m$
\item $\vert V(\tS) \vert \sim (1-\epsilon) m$
\item Any vertex $v \in V(\tS)$ has $deg (v) \sim(1-\epsilon)^2 \frac{m}{2}$ in $\tS$.
\item For any $c$-tuple of vertices in $V(S)$, where $c\leq d$, at least $r(c) = \frac{\epsilon^{c+1}m}{2(c^2+1)}$ of the disjoint sets $W_l$ guaranteed by Lemma~\ref{lemma:stars} lie entirely in $R$.
\end{enumerate}
\end{lemma}

\begin{proof}
Consider a set $R \subset V$ such that each vertex is chosen randomly
and independently with probability $\epsilon$.
We will use the following form of Chernoff bound (inequality (2.9) in~\cite{JLR}).
We will then fix an $R$ that has all of the desired properties.
\begin{theorem}[\cite{JLR}]\label{chernoff1}
For $\delta \leq 3/2$ and binomially distributed random variable $X$,
$$\mathbb{P}\left( \vert X - \mathbb{E} (X) \vert \geq \delta \mathbb{E} (X) \right) \leq 2 \exp \left( -\frac{\delta^2}{3} \mathbb{E}(X)\right).$$
\end{theorem}

For a reservoir $R$ selected randomly as described above, the expected number of vertices in $R$ is $\epsilon m$.
Since vertices are selected independently in $R$, we can apply Theorem~\ref{chernoff1}.
Almost surely the number of vertices chosen will be close to the expectation, so
$$\vert R \vert \sim \epsilon m \hspace*{1cm} \text{ and } \hspace*{1cm} \vert V(\tS) \vert \sim (1-\epsilon)m,$$
showing that (1) and (2) of the lemma hold with probability $1-o(1)$.

Considering (3), note that any vertex $v \in V(\tS)$ has degree $\frac{m-1}{2}$ in $S$.
For every edge containing $v$ in $S$, the edge contains two other vertices.
The probability that each of these vertices is in $\tS$ is $1-\epsilon$, so the expected number of edges incident to $v$ contained totally in $\tS$ is $(1-\epsilon)^2 \frac{m-1}{2}$.
Again by Theorem~\ref{chernoff1}, the degree of $v$ in $\tS$ is close to the expectation, so 
$$deg_{\tS} (v) \sim(1-\epsilon)^2 \frac{m}{2}.$$

Finally we consider (4) of the lemma and fix any $c$-tuple $\{v_1, \ldots, v_c\}$ of vertices in $S$ for some $c\leq d$.
By Lemma~\ref{lemma:stars}, there are $s(c) = \frac{m}{c^2+1}$ sets $W_l$, $l=1,\ldots,s(c)$, that are pairwise disjoint.
We need to show that $\frac{\epsilon^{c+1}m}{2(c^2+1)}= r(c)$ of these sets $W_l$ are subsets of $R$.
Let $X_l$ be the event that all $c+1$ vertices of $W_l$ are in $R$.
Clearly $\mathbb{P}(X_l) = \epsilon^{c+1}$, and if 
$$X = \sum_{l=1}^{s(c)} I(X_l)$$
(where $I(X_l)$ is an indicator random variable), then 
$$\mathbb{E}( X) =  \frac{\epsilon^{c+1}m}{c^2+1} = 2r(c).$$
Since the sets $W_l$, $1\leq l \leq \frac{m}{c^2+1}$, are disjoint, the events $X_l$ are independent, so we may apply Theorem~\ref{chernoff1} with $\delta=1/2$.
Then
$$\mathbb{P}\left( X \leq r(c) \right) \leq 2 \text{exp}\left( \frac{-r(c)}{6} \right).$$
So if $Y = Y(v_1,\ldots, v_c)$ is the event that for a fixed $v_1, \ldots, v_c$, there are fewer than $r(c)$ sets $W_l$ in $R$, then the probability of $Y$ is exponentially small in $m$ (since $r(c)$ grows with $m$).

Consequently the probability of $Y(v_1, \ldots, v_c)$ happening for any choice of $v_1,\ldots,v_c$ where $c$ can be any integer between 1 and $d$ is bounded by
$$\sum_{c=1}^d \binom{m}{c}\text{exp}\left( \frac{-r(c)}{4} \right) = o(1).$$
Thus with probability $1-o(1)$, for every $c$-tuple in $V(S)$, there are at least $r(c)$ pairwise-disjoint sets $W_l$ contained in $R$.

Since with probability $1-o(1)$, $R$ has each of the four properties listed by Lemma~\ref{lemma:random}, we can fix a set $R$ that has all four properties, completing the proof.

\end{proof}

\section{Embedding the Subhypertrees}

Let $P$ be as defined in Equation~\ref{Psum}.

\begin{lemma}\label{lemma:matching}
Let $\tS$ be an induced subhypergraph of a Steiner triple system $S$ on $m\geq(1+\mu)n$ vertices, such that $\vert V(\tS) \vert \sim (1-\epsilon) m$ and every vertex $v \in V(\tS)$ has $deg (v) \sim(1-\epsilon)^2 \frac{m}{2}$.
Then $\tS$ contains a copy of $P$.
\end{lemma}

Recall that
$$\vert \cP \vert = l.$$
For the proof of Lemma~\ref{lemma:matching}, we find it convenient for $l / {k3^k}$ to be an integer, where $k$ is as in Lemma~\ref{lemma:sawing}.
If not, add isolated vertices to $\cP$ (and to $P$), increasing $l$, until ${k3^k}$ divides $l$.
If we can find a copy of this larger $P$ in $\tS$, then surely we can find a copy of the original $P$ in $\tS$.

Consider a partition of $\cP$ with each partition class $\cC_i$ consisting of those members of $\cP$ that are pairwise isomorphic hypertrees.
Let
$$T_i= \text{isomorphism type of } \cC_i$$
and let
$$l_i = \vert \cC_i \vert.$$
Denoting the number of isomorphism classes by $t$, then
$$\sum_{i=1}^t l_i= \vert \cP \vert = l.$$
P\'olya (see~\cite{LPV} and~\cite{Pol}) showed an upperbound on the number of isomorphism classes of a tree on $k$ vertices.
Specifically,
$$t<3^k.$$
Recall that by assumption $k3^k \ll n.$
Since $ \frac{n}{k+3} \leq l$ (cf. Lemma~\ref{lemma:sawing}), we see that $t \ll l$, and in fact
$$ t < {k3^k} \ll l$$
for $n$ large.
In order to find the desired embedding of $P$ in $\tS$, we will first consider a "small" (with size independent of $n$) forest consisting of trees which form a "statistical sample" of $\cP$.
Next, we will find an almost perfect packing of vertex-disjoint copies of the small forest in $\tS$.
Finally, we will show that among the union of the copies of the small forest in $\tS$, there is a copy of $P$, meaning $P$ is a subhypergraph of $\tS$.

\begin{proof}\label{proof:matching}

We want to select a sampling of hypertrees from $\cP$ that has representatives from each partition class in proportion with the size of the class.

We first construct a forest consisting of about ${k3^k}$ hypertrees from $\cP$.
Let
$$\lambda_i = \Big\lceil \frac{{k3^k}l_i}{l}\Big\rceil,$$ so that if we were to choose ${k3^k}$ hypertrees randomly and independently from $\cP$, we would expect about $\lambda_i$ hypertrees of type $T_i$.
Consider the forest 
$$F = \bigvee_{i=1}^t \lambda_i T_i,$$
containing $\lambda_i$ vertex-disjoint hypertrees from class $T_i$, $i=1,\ldots, t$.
Let
$$\sum_{i=1}^t \lambda_i = \lambda > k3^k$$
be the number of connected components in the forest $F$.  In the remaining part of the proof we will show the following.

\begin{claim}\label{claim:forest}
The hypergraph $\tS$ contains $ \frac{l}{{k3^k}} $ vertex-disjoint copies of $F$.
\end{claim}

Before we establish the claim, observe that the claim immediately implies Lemma~\ref{lemma:matching}.
Indeed by Claim~\ref{claim:forest}, the number of vertex-disjoint copies of $T_i$ in $\tS$ is
$$\frac{l}{{k3^k}} \lambda_i =  \frac{l}{{k3^k}} \Big\lceil \frac{{k3^k}l_i}{l} \Big\rceil \geq l_i$$
for each $i=1,\ldots, t$.
Since $P$ contains exactly $l_i$ vertex-disjoint copies of $T_i$, $P$ is contained in $\tS$.
\end{proof}

\begin{proof}[Proof of Claim~\ref{claim:forest}]
We look for vertex-disjoint embeddings of $F$ in $\tS$.
First, we establish two upper bounds on the size of $\vert V(F) \vert$, which we denote by $r$.
\begin{prop}\label{prop:r}
For $\vert V(F) \vert = r$, the following holds.
\begin{enumerate}
\item $r \leq k(k+4)3^k$, and
\item $r \leq \frac{{k3^k} n}{l}\left( 1+\frac{\mu}{2}\right)$.
\end{enumerate}
\end{prop}
\begin{proof}[Proof of Proposition~\ref{prop:r}]
Recall that $t \leq 3^k$.
By Lemma~\ref{lemma:sawing}, $\vert V(T_i)\vert\leq k$, $i=1,\ldots,t$, and $l = \vert \cP \vert \geq  \frac{n}{k+3}$.
Also, $\vert V(P) \vert < n$.
%since the hypertrees that compose $P$ are all vertex-disjoint subhypertrees of $T$ and $n=\vert V(T) \vert$.
Then
\begin{align*}
r &= \sum_{i=1}^t \lambda_i \vert V(T_i)\vert = \sum_{i=1}^t \Big\lceil \frac{{k3^k} l_i}{l}\Big\rceil \vert V(T_i)\vert \leq \frac{{k3^k}}{l}\sum_{i=1}^t \ l_i\vert V(T_i)\vert + \sum_{i=1}^t \vert V(T_i)\vert\\
&\leq \frac{{k3^k}}{l} \sum_{i=1}^l\vert V(P_i)\vert + kt < \frac{{k3^k}}{l}\vert V(P)\vert + k3^k \leq \frac{k3^k}{n/(k+3)}n + k3^k \\
&\leq k(k+4)3^k,
\end{align*}
which proves $(1)$ of the proposition.

For the proof of $(2)$, we will first observe that $l \leq \frac{\mu n}{2}$ or equivalently that ${k3^k} \leq \frac{{k3^k} \mu n}{2 l}$.
This follows from~(\ref{sawing1}) and~(\ref{sawing3}) of Lemma~\ref{lemma:sawing}, which implies that $l \leq \frac{2d^2}{k}n$ and from the hierarchy Inequality~\ref{eq:hier} by which
$$\frac{2d^2}{k} \leq \frac{\mu}{2}$$
for $k \geq k_0(\mu, d)$.
Now applying in part our estimate from the proof of $(1)$ above, we infer that
\begin{align*}
r &< \frac{{k3^k}}{l}\vert V(P)\vert + k3^k \leq \frac{{k3^k}}{l}n + \frac{{k3^k} \mu n}{2 l}
\leq  \frac{{k3^k} n}{l} \left( 1 + \frac{\mu }{2}\right).
\end{align*}
\end{proof}

Now consider an auxiliary hypergraph $A$ so that
\begin{equation*}\tag{$\star$}
  \begin{array}{ccl}
    V(A) & = & V(\tilde{S}),\text{ and}\\
    E(A) & = & \{R \in \binom{V(\tS)}{r}: \tS[R] \text{ contains a copy of } F \}.
  \end{array}
\end{equation*}
In order to find vertex-disjoint copies of $F$ in $\tS$, we look for a matching in $A$.
To that end, we wish to apply the following theorem, where the co-degree $deg_A(x,y)$ of any $x,y \in V(A)$ is the number of edges shared by both $x$ and $y$.

\begin{theorem}[\cite{FR}]\label{thm:FR}
Suppose $A$ is an $r$-uniform hypergraph on $V$ which, for some $D >1$, has the following two properties:
\begin{enumerate}
\item $deg_A(x) = D(1+o(1))$ for all $x \in V$, where $o(1)\to 0$ as $\vert V \vert \to \infty$.
\item $deg_A(x,y) < D/(\log \vert V \vert)^4$ for all $x,y \in V$.
\end{enumerate}
Then $A$ contains at least $\frac{\vert V \vert(1-o(1))}{r}$ pairwise disjoint edges.
\end{theorem}

\begin{note}\label{pipspenc}
This theorem was subsequently extended and improved in a number of papers (e.g.~\cite{AKS}, \cite{KK}, \cite{PS}), but for the purposes here, it is sufficient to use this form.
\end{note}

We will show the following result (which is proved on the next page), where
$$f=\vert E(F) \vert.$$
\begin{claim}\label{claim:degcodeg}
There exists a constant $c=c(F)$ such that
$$D = c \vert V(A) \vert^{\lambda+f-1}$$
satisfies $(1)$ and $(2)$ of Theorem~\ref{thm:FR} for the auxiliary hypergraph $A$ defined in $(\star)$.
\end{claim}

Thus by Theorem~\ref{thm:FR}, there is a nearly perfect matching in $A$, implying a nearly perfect packing of vertex-disjoint copies of $F$ in $\tS$.
Specifically Theorem~\ref{thm:FR} says one can pack at least
$$\frac{\vert V(A) \vert (1-o(1))}{r}$$
vertex-disjoint copies of $F$ into $\tS$.

To prove Claim~\ref{claim:forest}, it remains to show that
$$\frac{\vert V(A) \vert (1-o(1))}{r} \geq \frac{l}{{k3^k}}.$$
Recall that $\epsilon \ll \mu$ and $\vert V(A) \vert \geq (1-\epsilon)(1+\mu)n$.
By part (2) of Proposition~\ref{prop:r}, it follows that
$$\frac{\vert V(A) \vert(1-o(1))}{r} \geq \frac{(1-\epsilon)(1+\mu) n (1-o(1))}{\frac{{k3^k} n}{l}\left( 1+\frac{\mu}{2}\right)} \geq \frac{(1+\frac{\mu}{2}) n }{\frac{{k3^k} n}{l}\left( 1+\frac{\mu}{2}\right)} = \frac{l}{{k3^k}}.$$
Therefore there are $\frac{l}{{k3^k}}$ vertex-disjoint copies of $F$ in $\tS$.
\end{proof}

For the proof of Claim~\ref{claim:degcodeg}, we formalize the definition of an embedding.
\begin{definition}
Let $V(F) = \{1,2,\ldots, r\}$ and let $R\subset V(\tS)$ be a subset of $V(\tS)$ with labeled vertices $\{v_1, v_2, \ldots, v_r\}$.
We say a function $\psi:F\to \tS$ is an \emph{embedding} of $F$ into $\tS$ if $\psi(i)=v_i$ for $i=1, \ldots, r$ and if for all $\{i,j,h\} \in E(F)$, $\{v_i, v_j, v_h\} \in E(\tS)$.
\end{definition}

\begin{proof}[Proof of Claim~\ref{claim:degcodeg}]
To count $\deg_A(x)$ for any $x\in V(A)$, we first count the number of embeddings from $F$ to $\tS$ that map some vertex in $V(F)$ to $x$.
Fix a labeling $\{1,2,\ldots, r\}$ of $V(F)$ and fix any $x\in V(\tS)$.
Then let
$$E_x = \{\psi: F\to \tS \text{ such that there is some }i\in V(F) \text{ with }  \psi(i)=x\}$$
be the set of all embeddings of $F$ into $\tS$ where $x$ is in the image of $\psi$.
Let
\begin{equation}\label{degAx}
D_x = \{ R\in \binom{V(\tS)}{r}: \text{ there exist } \psi\in E_x \text{ with } \psi(V(F)) = R\}
\end{equation}
and see that
$$deg_A(x) = \vert D_x \vert.$$

In order to determine $\vert D_x \vert$, we will find the cardinality of $E_x$.
With $i$ fixed, consider all embeddings $\psi: F\to \tS$ with $i\to x$.
For simplicity, first consider the case that $F$ consists of a single tree.
Consider an ordering of the edges $e_t\in E(F)$, $t= 1, \ldots, f$, satisfying 
\begin{equation}\label{edgeorder}
  \begin{array}{c}
    e_1 = \{i,j,h\} \text{ for some vertices } j \text{ and } h, \text{ and }\\
    \vert e_{t+1} \cap \left( \bigcup_{s\leq t} e_s \right) \vert =1.
  \end{array}
\end{equation}
Recall that by assumption, for all $v\in V(F)$,
$$deg_{\tS} (v) \sim (1 -\epsilon)^2 \frac{m}{2} \sim (1-\epsilon)\frac{\vert V(\tS) \vert}{2}.$$
For edge $e_1$, there are $(1-\epsilon)\frac{\vert V(\tS) \vert}{2}$ edges incident to $x$ in $\tS$ onto which to map $e_1$.
Choosing one of these edges, say $\{x,y_1, y_2\}$, there are two ways to map $\{i,j,h\}$ to $\{x,y_1, y_2\}$ with $i \to x$.
Therefore there are $(1-\epsilon)\vert V(\tS) \vert$ ways to map $e_1$ into $\tS$.
For $t>1$, one vertex of $e_t$ has already been mapped into $\tS$.
Consequently, similarly as in the $t=1$ case, there are $(1-o(1))(1-\epsilon)\vert V(\tS) \vert$ ways to embed $e_t$ for $t=2, \ldots, f$ into $\tS$.
Therefore we have $(1-o(1))(1-\epsilon)^f \vert V(\tS) \vert^f$
embeddings with $i\to x$.
Since $i$ can be chosen in $r= \vert V(F) \vert$ ways, if $F$ is a simple tree,
\begin{equation}\label{Extree}
\vert E_x \vert \sim r(1-\epsilon)^f \vert V(\tS) \vert^f.
\end{equation}

Now consider $F$ as a forest of $\lambda$ disjoint trees.
As above, map some vertex $i$ to $x$ and embed all of the edges in the same component of $i$ into $\tS$.
For each of the other $\lambda-1$ components of $F$, embed one of its vertices to some unused vertex in $\tS$.  There are
\begin{equation}\label{embedseeds}
(1-o(1))\vert V(\tS)\vert^{\lambda-1}
\end{equation}
ways to choose these vertices.
As with the first component, map the rest of the edges of $F$ into $\tS$ to form an embedding of $V(F)$ to some subset $R \subset V(\tS)$.
Combining~\ref{Extree} and~\ref{embedseeds}, we infer that
\begin{equation}\label{sizeEx}
\vert E_x \vert \sim r(1-\epsilon)^f \vert V(\tS) \vert^{f+\lambda-1}.
\end{equation}

Some embeddings $\psi$ in $E_x$ may map onto the same vertex sets but different edge sets in $\tS$.
In order to find $deg_A(x)$, we make sure that each vertex set $R$ with $x\in R$ inducing a copy of $F$ is counted precisely once.
To this end, let
$$R_x = \{ \psi \in E_x: \text{ there exists } \psi'\in E_x \text{ with } \psi(V(F))=\psi'(V(F))\text{ but } \psi(E(F)) \neq\psi'(E(F))\}.$$
For any $\psi, \psi' \in R_x$, there must be some edge in $\tS[\psi(V(F))]$ not in $\psi(E(F))$, because this edge is in $\psi'(E(F))$.
In the following claim, we show that $R_x$ makes up a small portion of $E_x$ by counting how many embeddings in $E_x$ induce no extra edge in $\tS$. 
This implies that the number of embeddings $\psi$ for which there is a $\psi'$ with the property above is negligible.

\begin{claim}\label{Rxsmall}
$$\vert E_x \setminus R_x \vert \sim \vert E_x \vert$$
\end{claim}
\begin{proof}
The argument will be similar to that of the proof of \ref{sizeEx} with one additional constraint.
In order to determine $\vert E_x \setminus R_x \vert$, fix $i$ and consider all $\psi\in E_x \setminus R_x$ with $i\to x$.
First consider the case that $F$ consists of a single tree.
As before, consider an ordering of the edges $e_t\in E(F)$ as in \ref{edgeorder}.
For edge $e_1$, there are $(1-\epsilon)\vert V(\tS) \vert$ ways to map $e_1$ into $\tS$ with $i\to x$.
For $t>1$, similarly as before, one vertex of $e_t$ has already been mapped into $\tS$.
To choose which two vertices in $\tS$ to map the other two vertices of $e_t$ to, we have to avoid creating an \emph{unwanted triple} in the image of $F$:
i.e. a triple $\{ a_1, a_2, a_3 \} \subset V(F)$ with $\{ a_1, a_2, a_3 \} \not\in E(F)$ while $\{ \psi(a_1), \psi(a_2), \psi(a_3) \}\in E(\tS)$.
By avoiding unwanted edges, $\tS[\psi(V(F))]$ will equal $\psi(E(F))$.
For any two vertices $a_1, a_2$ with
$$a_1,a_2 \in \bigcup_{s=1}^{t-1}e_s,$$
these two vertices have already been mapped into $\tS$ with the first $t-1$ edges.
We need to select the image of the remaining two vertices of $e_t$ different from the vertex $b\in V(\tS)$ for which $\{ \psi(a_1), \psi(a_2), b\} \in E(\tS)$.
Since there are at most $\binom{r}{2}$ pairs $a_1, a_2$, at most $\binom{r}{2}$ vertices $b$ in $\tS$ are forbidden to be selected.
Since by Proposition~\ref{prop:r} $r \leq k(k+4)3^k \ll \vert V(\tS) \vert$, there are still $(1-o(1))(1-\epsilon)\vert V(\tS) \vert$ ways to embed $e_t$ for $t=2,\ldots, f$ into $\tS$.
As before, since $i$ can be chosen in $r$ ways, if $F$ is a tree,
$$\vert E_x \setminus R_x \vert \sim r(1-\epsilon)^f \vert V(\tS) \vert^f.$$

If $F$ is a forest, then again proceed as before.
When mapping the first vertex of each component into $\tS$, we still must avoid creating unwanted edges in $\tS$.
The number of vertices that are forbidden to be selected here is still small though, because the number of components $\lambda \ll \vert V(\tS) \vert$.
Then as before, if $F$ is a forest,
$$\vert E_x \setminus R_x \vert \sim r(1-\epsilon)^f \vert V(\tS) \vert^{f+\lambda-1} \sim \vert E_x \vert.$$
\end{proof}

It still may be that for two distinct embeddings $\psi,\psi'\in E_x\setminus R_x$, the images $\psi(V(F))=\psi'(V(F))$ and $\psi(E(F))=\psi'(E(F))$.
Fix any copy of $F$ in $\tS$.
For each labeling of $V(F)$ that gives an automorphism of $F$, there is a distinct embedding $\psi\in E_x$ onto the copy of $F$.
Let the number of hypergraph automorphisms of $F$ be called $\vert Aut(F)\vert$.
Then there are $\vert Aut(F)\vert$ distinct embeddings in $E_x$ onto any fixed copy of $F$.
Since $D_x$ counts unlabeled sets containing $x$ that induce a copy of $F$, we infer that
$$\frac{\vert E_x\setminus R_x \vert}{\vert Aut(F) \vert} \leq \vert D_x \vert \leq \frac{\vert E_x\vert}{\vert Aut(F) \vert}.$$
This along with Claim~\ref{Rxsmall} and Equation~\ref{sizeEx} implies that
$$deg_A(x) = \vert D_x \vert \sim \frac{r(1-\epsilon)^f \vert V(\tS) \vert^{f+\lambda-1}}{\vert Aut(F) \vert},$$
proving $(1)$ of Claim~\ref{claim:degcodeg} with 
$$c = \frac{r(1-\epsilon)^f}{\vert Aut(F) \vert}.$$

To find $deg_{A}(x,y)$ for any $x,y \in V(A)$, proceed as before.
Let
$$E_{x,y} = \{\psi: F\to \tS \text{ such that there are some }i,j\in V(F) \text{ with }  \psi(i)=x \text{ and } \psi(j)=y\}$$
be the set of all embeddings of $F$ into $\tS$ with $i\to x$ and $j\to y$.
Let
\begin{equation}\label{degAx}
D_{x,y} = \{ R\in \binom{V(\tS)}{r}: \text{ there exist } \psi\in E_{x,y} \text{ with } \psi(V(F)) = R\}
\end{equation}
and see that
$$deg_A(x,y) = \vert D_{x,y} \vert.$$

To find the cardinality of $E_{x,y}$, we follow the same procedure as for $E_{x}$, except that some vertex in $F$ must be mapped to $y$ by all $\psi \in E_{x,y}$.
Fix some $i,j\in V(F)$ and first count the embeddings  $\psi \in E_{x,y}$ with $i\to x$ and $j\to y$.
Consider two possible cases.

\begin{enumerate}
\item[Case 1)] Suppose $j$ is in the same component of $F$ as $i$.
Call the component $C_i$.
As before, give an ordering to the edges $e_t\in E(C_i)$ such that
$$e_1 = \{i,g,h\} \text{ for some vertices } g \text{ and } h \text{, and}$$
$$\big \vert e_{t+1} \cap \left( \bigcup_{s\leq t} e_s \right) \big\vert =1.$$
Let $e_j$ be the first edge in the order that contains $j$ as a vertex, so $e_j = \{ a,b,j\}$ for some vertices $a$ and $b$.
For edge $e_1$, as before, there are $(1-\epsilon)\vert V(\tS) \vert$ ways to map $e_1$ into $\tS$ with $i\to x$.
For $2 \leq t < j$, similarly as before, there are $(1-o(1)(1-\epsilon)\vert V(\tS) \vert$ ways to map $e_t$ into $\tS$.
For $e_j$, either $a$ or $b$ has already been assigned an image in $\tS$, and since $j\to y$, there is only one way to map $e_j$ into $\tS$.
For $e_t$, $t>j$, once again there are $(1-o(1)(1-\epsilon)\vert V(\tS) \vert$ ways to map $e_t$ into $\tS$.
Embed the other $\lambda-1$ components of $F$ into $\tS$ as before, so all $f$ edges are mapped into $\tS$.
Then there are $(1-o(1))(1-\epsilon)^{f-1} \vert V(\tS) \vert^{f+\lambda-2}$
embeddings with $i\to x$ and $j\to y$.
Since $i$ and $j$ can be chosen in $r(r-1)$ ways,
$$
\vert E_{x,y} \vert \sim r^2(1-\epsilon)^{f-1} \vert V(\tS) \vert^{f+\lambda-2}.
$$

\item[Case 2)] Suppose $j$ is not in the same component of $F$ as $i$.
Map the component containing $i$ as before, with $i\to x$.
Next map the component containing $j$ in the same way, with $j\to x$.
For each of the other $\lambda-2$ components of $F$, embed one of its vertices to some unused vertex in $\tS$.
There are
\begin{equation}\label{embedseedsnoty}
(1-o(1))\vert V(\tS)\vert^{\lambda-2}
\end{equation}
ways to choose these vertices.
As with the first two components, map the rest of the edges of $F$ into $\tS$ to form an embedding of $F$.
Since each edge of $F$ can be mapped to $\tS$ in $(1-o(1))(1-\epsilon)\vert V(\tS) \vert$ ways, and since there are $r(r-1)$ ways to choose $i$ and $j$, applying~\ref{embedseedsnoty} we infer that
\begin{equation}
\vert E_x \vert \sim r^2(1-\epsilon)^{f} \vert V(\tS) \vert^{f+\lambda-2}.
\end{equation}
\end{enumerate}

In either case,
\begin{equation}\label{sizeExy}
\vert E_{x,y} \vert \lesssim r^2(1-\epsilon)^{f-1} \vert V(\tS) \vert^{f+\lambda-2}.
\end{equation}
By the same argument that showed
$$\vert D_x \vert \sim \frac{r(1-\epsilon)^f \vert V(\tS) \vert^{f+\lambda-1}}{\vert Aut(F) \vert},$$
it can be shown that
$$ deg_A(x,y) = \vert D_{x,y} \vert \lesssim \frac{r^2(1-\epsilon)^{f-1} \vert V(\tS) \vert^{f+\lambda-2}}{\vert Aut(F) \vert}.$$
Since $\vert V(\tS) \vert = \vert V(A) \vert \sim (1-\epsilon)m$ by assumption, and since $r$ is bounded by a fixed value by Proposition~\ref{prop:r}, 
$$deg_A(x,y) < \frac{c \vert V(A) \vert^{f+\lambda+1}}{(\log \vert V(A) \vert)^4}$$
for the same value of $c$ as above.
This proves $(2)$ of Claim~\ref{claim:degcodeg}.
\end{proof}

\section{Proof of Theorem~\ref{thm:main}}

\begin{proof}\label{proof:main}
Let $T$ by any subdivision tree on $n$ vertices of bounded degree $d$ and let $S$ be any Steiner triple system on $m \geq n(1+\mu)$ vertices, where $n$ is large.
Fix constants $\epsilon$ and $k$ so that they fit in the hierarchy described in Inequality \ref{eq:hier}.
Specifically, choose $k$ so that
\begin{equation}\label{sizek}
k \geq \frac{8d^5}{\epsilon^{d+1}}.
\end{equation}

First, recall that Lemma~\ref{lemma:sawing} guarantees a decomposition of $T$ into families with certain properties.
Namely, $T$ is decomposed into a set $\cP$ of subhypertrees, a set $\cE$ of stars, and a set $I$ of independent vertices.

Second, recall that Lemma~\ref{lemma:random} guarantees a subset $R \subset V(S)$ called the \emph{reservoir} with certain properties.
Let
$$\tS = S[V(S) \setminus R].$$

Third, let $P$ be as defined in Equation~\ref{Psum}.
Lemma~\ref{lemma:matching} guarantees that there is a copy of $P$ in $\tS$, so that the hypergraph embedding
$$f: P \to \tS$$
exists.

Finally, it remains to show that $\cE$ and $I$ can be embedded in $S$ in such a way that the original configuration of $T$ is restored.
Recall that a star $E_j \in \cE$ has the form
$$E_j = \{ v_{j,i}, w_{j,i}, u_j\} \text{ where } 1\leq i \leq {c_j} = deg(u_j).$$
To simplify the notation throughout the rest of the proof, we will drop the index $j$ and will write $v_i = v_{j,i}$, $w_i = w_{j,i}$, $u = u_j$, and $c=c_j$ when referring to vertices of $E_j$, whenever it is clear from the context that $j$ is fixed.

To embed the stars belonging to $\cE$, first take the star $E_1 = \{ v_{i}, w_{i}, u\}$, $1\leq i\leq c$.
Referring to Figure~\ref{embedreminder}, recall how the vertices of a star are labeled.
By~(\ref{sawing6}) of Lemma~\ref{lemma:sawing},
$$w_{1}, w_{2},\ldots, w_{c}, u \in I, \text{ and }$$
$$v_1, \ldots, v_c \in V(P),$$
so the vertices $v_i$ already have an image in $S$ by $f$.
Specifically, all $f(v_i)$ are in $\tS$.  
By Lemma~\ref{lemma:stars}, for the $c$-tuple of vertices $f(v_{1}), \ldots, f(v_{c})$, there are at least $s(c)=\frac{m}{c^2+1}$ stars of the form 
$$S_l = \{ \{f(v_i),w_i^{(l)},u^{(l)}\},\hspace*{0.15cm} 1 \leq i \leq c\}, \hspace*{0.5cm} l=1,\ldots,s(c)$$
in $S$ such that the sets
$$W_l = \{w_1^{(l)},\ldots,w_c^{(l)}, u^{(l)}\}, \hspace*{0.5cm} l=1,\ldots,s(c)$$
are pairwise disjoint subsets of $V(S)$.
Further, Lemma~\ref{lemma:random} guarantees that at least $r(c)=\frac{\epsilon^{c+1}m}{2(c^2+1)}$ of the sets $W_l$ lie in the reservoir $R$.
As an example, Figure~\ref{embedding1} shows just two such sets $W_1$ and $W_2$, where a dashed line segment represents a hyperedge in $S$.

\begin{figure}[h]
    \begin{minipage}{0.33\textwidth}
    	\centering
	\definecolor{zzttqq}{rgb}{0,0,0} %{0.6,0.2,0.}
\definecolor{qqqqff}{rgb}{0,0,0} %{0.,0.,1.}
\begin{tikzpicture}[line cap=round,line join=round,>=triangle 45,x=1.0cm,y=0.7cm]
%\clip(-7.23567731300529,-6.720316430435862) rectangle (19.44456628798421,6.383697019607706);
\fill[line width=2.pt,dash pattern=on 4pt off 4pt,color=zzttqq,fill opacity=0] (6.,1.) -- (4.2,-1.) -- (4.8,-1.) -- cycle;
\fill[line width=2.pt,dash pattern=on 4pt off 4pt,color=zzttqq,fill opacity=0] (6.,1.) -- (5.7,-1.) -- (6.3,-1.) -- cycle;
\fill[line width=2.pt,dash pattern=on 4pt off 4pt,color=zzttqq,fill opacity=0] (6.,1.) -- (7.2,-1.) -- (7.8,-1.) -- cycle;
\fill[line width=2.pt,dash pattern=on 4pt off 4pt,color=zzttqq,fill opacity=0] (6.,1.) -- (5.4,1.) -- (4.1,3.) -- cycle;
\fill[line width=2.pt,color=zzttqq,fill=zzttqq,fill opacity=0.10000000149011612] (4.1,3.) -- (4.3641429644735465,1.0254659940610364) -- (3.806023398089517,1.0254659940610364) -- cycle;
\fill[line width=2.pt,color=zzttqq,fill=zzttqq,fill opacity=0.10000000149011612] (4.1,3.) -- (2.832891846445569,1.0254659940610364) -- (2.3168531886162764,1.0216671244288105) -- cycle;
\fill[line width=2.pt,color=zzttqq,fill=zzttqq,fill opacity=0.10000000149011612] (4.1,3.) -- (3.509291958589649,3.0008696189206994) -- (3.8,5.) -- cycle;
\fill[line width=2.pt,color=zzttqq,fill=zzttqq,fill opacity=0.10000000149011612] (4.8,-1.) -- (3.8657942712621005,-2.9859105973373743) -- (4.41898751506418,-2.998203780532976) -- cycle;
\fill[line width=2.pt,color=zzttqq,fill=zzttqq,fill opacity=0.10000000149011612] (4.8,-1.) -- (5.,-3.) -- (5.6114262850375525,-3.010496963728578) -- cycle;
\fill[line width=2.pt,color=zzttqq,fill=zzttqq,fill opacity=0.10000000149011612] (7.8,-1.) -- (7.,-3.) -- (7.60292196272504,-2.998203780532976) -- cycle;
\fill[line width=2.pt,color=zzttqq,fill=zzttqq,fill opacity=0.10000000149011612] (7.8,-1.) -- (8.15611520652712,-3.010496963728578) -- (8.758481183111607,-3.010496963728578) -- cycle;
\draw [line width=2.pt,dash pattern=on 4pt off 4pt,color=zzttqq] (6.,1.)-- (4.2,-1.);
\draw [line width=2.pt,dash pattern=on 4pt off 4pt,color=zzttqq] (4.2,-1.)-- (4.8,-1.);
\draw [line width=2.pt,dash pattern=on 4pt off 4pt,color=zzttqq] (4.8,-1.)-- (6.,1.);
\draw [line width=2.pt,dash pattern=on 4pt off 4pt,color=zzttqq] (6.,1.)-- (5.7,-1.);
\draw [line width=2.pt,dash pattern=on 4pt off 4pt,color=zzttqq] (5.7,-1.)-- (6.3,-1.);
\draw [line width=2.pt,dash pattern=on 4pt off 4pt,color=zzttqq] (6.3,-1.)-- (6.,1.);
\draw [line width=2.pt,dash pattern=on 4pt off 4pt,color=zzttqq] (6.,1.)-- (7.2,-1.);
\draw [line width=2.pt,dash pattern=on 4pt off 4pt,color=zzttqq] (7.2,-1.)-- (7.8,-1.);
\draw [line width=2.pt,dash pattern=on 4pt off 4pt,color=zzttqq] (7.8,-1.)-- (6.,1.);
\draw [line width=2.pt,dash pattern=on 4pt off 4pt,color=zzttqq] (6.,1.)-- (5.4,1.);
\draw [line width=2.pt,dash pattern=on 4pt off 4pt,color=zzttqq] (5.4,1.)-- (4.1,3.);
\draw [line width=2.pt,dash pattern=on 4pt off 4pt,color=zzttqq] (4.1,3.)-- (6.,1.);
\draw [line width=2.pt,color=zzttqq] (4.1,3.)-- (4.3641429644735465,1.0254659940610364);
\draw [line width=2.pt,color=zzttqq] (4.3641429644735465,1.0254659940610364)-- (3.806023398089517,1.0254659940610364);
\draw [line width=2.pt,color=zzttqq] (3.806023398089517,1.0254659940610364)-- (4.1,3.);
\draw [line width=2.pt,color=zzttqq] (4.1,3.)-- (2.832891846445569,1.0254659940610364);
\draw [line width=2.pt,color=zzttqq] (2.832891846445569,1.0254659940610364)-- (2.3168531886162764,1.0216671244288105);
\draw [line width=2.pt,color=zzttqq] (2.3168531886162764,1.0216671244288105)-- (4.1,3.);
\draw [line width=2.pt,color=zzttqq] (4.1,3.)-- (3.509291958589649,3.0008696189206994);
\draw [line width=2.pt,color=zzttqq] (3.509291958589649,3.0008696189206994)-- (3.8,5.);
\draw [line width=2.pt,color=zzttqq] (3.8,5.)-- (4.1,3.);
\draw [line width=2.pt,color=zzttqq] (4.8,-1.)-- (3.8657942712621005,-2.9859105973373743);
\draw [line width=2.pt,color=zzttqq] (3.8657942712621005,-2.9859105973373743)-- (4.41898751506418,-2.998203780532976);
\draw [line width=2.pt,color=zzttqq] (4.41898751506418,-2.998203780532976)-- (4.8,-1.);
\draw [line width=2.pt,color=zzttqq] (4.8,-1.)-- (5.,-3.);
\draw [line width=2.pt,color=zzttqq] (5.,-3.)-- (5.6114262850375525,-3.010496963728578);
\draw [line width=2.pt,color=zzttqq] (5.6114262850375525,-3.010496963728578)-- (4.8,-1.);
\draw [line width=2.pt,color=zzttqq] (7.8,-1.)-- (7.,-3.);
\draw [line width=2.pt,color=zzttqq] (7.,-3.)-- (7.60292196272504,-2.998203780532976);
\draw [line width=2.pt,color=zzttqq] (7.60292196272504,-2.998203780532976)-- (7.8,-1.);
\draw [line width=2.pt,color=zzttqq] (7.8,-1.)-- (8.15611520652712,-3.010496963728578);
\draw [line width=2.pt,color=zzttqq] (8.15611520652712,-3.010496963728578)-- (8.758481183111607,-3.010496963728578);
\draw [line width=2.pt,color=zzttqq] (8.758481183111607,-3.010496963728578)-- (7.8,-1.);
\begin{scriptsize}
\draw [fill=qqqqff] (6.,1.) circle (2.5pt);
\draw[color=qqqqff] (6.306823073590519,1.062826693367364) node {u};
\draw [fill=qqqqff] (4.2,-1.) circle (2.5pt);
\draw[color=qqqqff] (3.7011987775393953,-0.9356618443805933) node {$w_{2}$};
\draw [fill=qqqqff] (4.8,-1.) circle (2.5pt);
\draw[color=qqqqff] (5.117848880246802,-0.9019320800304167) node {$v_{2}$};
\draw [fill=qqqqff] (5.7,-1.) circle (2.5pt);
\draw[color=qqqqff] (5.573200698974183,-1.3) node {$w_{3}$};
\draw [fill=qqqqff] (6.3,-1.) circle (2.5pt);
\draw[color=qqqqff] (6.332120396853152,-1.3) node {$v_{3}$};
\draw [fill=qqqqff] (7.2,-1.) circle (2.5pt);
\draw[color=qqqqff] (7.057310330381943,-1.3) node {$w_{4}$};
\draw [fill=qqqqff] (7.8,-1.) circle (2.5pt);
\draw[color=qqqqff] (8.153527671762674,-0.9356618443805933) node {$v_{4}$};
\draw [fill=qqqqff] (5.4,1.) circle (2.5pt);
\draw[color=qqqqff] (4.94920005849592,1.05439425227982) node {$w_{1}$};
\draw [fill=qqqqff] (4.1,3.) circle (2.5pt);
\draw[color=qqqqff] (4.392658946718011,3.162504524165851) node {$v_{1}$};
\draw [fill=qqqqff] (3.8,5.) circle (2.5pt);
%\draw[color=qqqqff] (4.4179562699806425,5.312777001489602) node {toward root};
\end{scriptsize}
\end{tikzpicture}
    	\caption{The dashed edges of star $E_j$ are removed during decomposition \label{embedreminder}}
    	\end{minipage}
    	\begin{minipage}{0.59\textwidth}
    		\centering
\definecolor{zzttqq}{rgb}{0,0,0} %{0.6,0.2,0.}
\definecolor{qqqqff}{rgb}{0,0,0} %{0.,0.,1.}
\begin{tikzpicture}[line cap=round,line join=round,>=triangle 45,x=1.0cm,y=1.0cm]
\clip(0.34079231360926787,-1.2) rectangle (23.100787761327332,9.91380104385483);
\fill[line width=2.pt,color=zzttqq,fill=zzttqq,fill opacity=0.10000000149011612] (4.1,3.) -- (4.3641429644735465,1.0254659940610364) -- (3.806023398089517,1.0254659940610364) -- cycle;
\fill[line width=2.pt,color=zzttqq,fill=zzttqq,fill opacity=0.10000000149011612] (4.1,3.) -- (2.832891846445569,1.0254659940610364) -- (2.3168531886162764,1.0216671244288105) -- cycle;
\fill[line width=2.pt,color=zzttqq,fill=zzttqq,fill opacity=0.10000000149011612] (4.1,3.) -- (3.509291958589649,3.0008696189206994) -- (3.8,5.) -- cycle;
\fill[line width=2.pt,color=zzttqq,fill=zzttqq,fill opacity=0.10000000149011612] (6.416661752659802,2.906503383945089) -- (5.155954180533691,1.0119923313931234) -- (5.763617847966775,1.011083270323857) -- cycle;
\fill[line width=2.pt,color=zzttqq,fill=zzttqq,fill opacity=0.10000000149011612] (6.416661752659802,2.906503383945089) -- (7.133417257810685,0.9632995699804647) -- (6.528157053461051,0.9632995699804647) -- cycle;
\fill[line width=2.pt,color=zzttqq,fill=zzttqq,fill opacity=0.10000000149011612] (9.745592876582792,2.8587196836016964) -- (8.996981571202982,0.8836600694081441) -- (9.634097575781544,0.8995879695226082) -- cycle;
\fill[line width=2.pt,color=zzttqq,fill=zzttqq,fill opacity=0.10000000149011612] (9.745592876582792,2.8587196836016964) -- (10.271213580360108,0.8677321692936799) -- (10.860545884595279,0.8518042691792158) -- cycle;
\draw [line width=2.pt,color=zzttqq] (4.1,3.)-- (4.3641429644735465,1.0254659940610364);
\draw [line width=2.pt,color=zzttqq] (4.3641429644735465,1.0254659940610364)-- (3.806023398089517,1.0254659940610364);
\draw [line width=2.pt,color=zzttqq] (3.806023398089517,1.0254659940610364)-- (4.1,3.);
\draw [line width=2.pt,color=zzttqq] (4.1,3.)-- (2.832891846445569,1.0254659940610364);
\draw [line width=2.pt,color=zzttqq] (2.832891846445569,1.0254659940610364)-- (2.3168531886162764,1.0216671244288105);
\draw [line width=2.pt,color=zzttqq] (2.3168531886162764,1.0216671244288105)-- (4.1,3.);
\draw [line width=2.pt,color=zzttqq] (4.1,3.)-- (3.509291958589649,3.0008696189206994);
\draw [line width=2.pt,color=zzttqq] (3.509291958589649,3.0008696189206994)-- (3.8,5.);
\draw [line width=2.pt,color=zzttqq] (3.8,5.)-- (4.1,3.);
\draw [line width=2.pt,color=zzttqq] (6.416661752659802,2.906503383945089)-- (5.155954180533691,1.0119923313931234);
\draw [line width=2.pt,color=zzttqq] (5.155954180533691,1.0119923313931234)-- (5.763617847966775,1.011083270323857);
\draw [line width=2.pt,color=zzttqq] (5.763617847966775,1.011083270323857)-- (6.416661752659802,2.906503383945089);
\draw [line width=2.pt,color=zzttqq] (6.416661752659802,2.906503383945089)-- (7.133417257810685,0.9632995699804647);
\draw [line width=2.pt,color=zzttqq] (7.133417257810685,0.9632995699804647)-- (6.528157053461051,0.9632995699804647);
\draw [line width=2.pt,color=zzttqq] (6.528157053461051,0.9632995699804647)-- (6.416661752659802,2.906503383945089);
\draw [line width=2.pt,color=zzttqq] (9.745592876582792,2.8587196836016964)-- (8.996981571202982,0.8836600694081441);
\draw [line width=2.pt,color=zzttqq] (8.996981571202982,0.8836600694081441)-- (9.634097575781544,0.8995879695226082);
\draw [line width=2.pt,color=zzttqq] (9.634097575781544,0.8995879695226082)-- (9.745592876582792,2.8587196836016964);
\draw [line width=2.pt,color=zzttqq] (9.745592876582792,2.8587196836016964)-- (10.271213580360108,0.8677321692936799);
\draw [line width=2.pt,color=zzttqq] (10.271213580360108,0.8677321692936799)-- (10.860545884595279,0.8518042691792158);
\draw [line width=2.pt,color=zzttqq] (10.860545884595279,0.8518042691792158)-- (9.745592876582792,2.8587196836016964);
\draw [line width=2.pt,dash pattern=on 3pt off 3pt,color=zzttqq] (4.1,3.)-- (4.771941869802165,6.288315043333342);
\draw [line width=2.pt,dash pattern=on 3pt off 3pt,color=zzttqq] (4.771941869802165,6.288315043333342)-- (4.973357758720024,7.597518321299435);
\draw [line width=2.pt,dash pattern=on 3pt off 3pt,color=zzttqq] (4.973357758720024,7.597518321299435)-- (5.304255290513649,6.2451544957080865);
\draw [line width=2.pt,dash pattern=on 3pt off 3pt,color=zzttqq] (5.304255290513649,6.2451544957080865)-- (6.416661752659802,2.906503383945089);
\draw [line width=2.pt,dash pattern=on 3pt off 3pt,color=zzttqq] (8.16873076525085,2.8905754838306246)-- (5.721473917557786,6.331475590958598);
\draw [line width=2.pt,dash pattern=on 3pt off 3pt,color=zzttqq] (5.721473917557786,6.331475590958598)-- (4.973357758720024,7.597518321299435);
\draw [line width=2.pt,dash pattern=on 3pt off 3pt,color=zzttqq] (4.973357758720024,7.597518321299435)-- (6.0811451477682485,6.51850463066804);
\draw [line width=2.pt,dash pattern=on 3pt off 3pt,color=zzttqq] (6.0811451477682485,6.51850463066804)-- (9.745592876582792,2.8587196836016964);
\draw [line width=2.pt] (5.2754815920968126,7.022044352962692) circle (1.2576772595439563cm);
\draw [line width=2.pt,dash pattern=on 3pt off 3pt,color=zzttqq] (4.1,3.)-- (8.46936211636572,6.676759971960644);
\draw [line width=2.pt,dash pattern=on 3pt off 3pt,color=zzttqq] (8.46936211636572,6.676759971960644)-- (9.577149505413944,7.554357773674179);
\draw [line width=2.pt,dash pattern=on 3pt off 3pt,color=zzttqq] (9.577149505413944,7.554357773674179)-- (8.800259648159344,6.389022987792273);
\draw [line width=2.pt,dash pattern=on 3pt off 3pt,color=zzttqq] (8.800259648159344,6.389022987792273)-- (6.416661752659802,2.906503383945089);
\draw [line width=2.pt,dash pattern=on 3pt off 3pt,color=zzttqq] (8.16873076525085,2.8905754838306246)-- (9.203091425995062,6.273928194124924);
\draw [line width=2.pt,dash pattern=on 3pt off 3pt,color=zzttqq] (9.203091425995062,6.273928194124924)-- (9.577149505413944,7.554357773674179);
\draw [line width=2.pt,dash pattern=on 3pt off 3pt,color=zzttqq] (9.577149505413944,7.554357773674179)-- (9.6203100530392,6.273928194124924);
\draw [line width=2.pt,dash pattern=on 3pt off 3pt,color=zzttqq] (9.6203100530392,6.273928194124924)-- (9.745592876582792,2.8587196836016964);
\draw [line width=2.pt] (9.174317727578226,6.964496956129016) circle (1.2700418438484615cm);
\draw [rotate around={-0.5541005153679632:(7.325607604296451,6.986077229941646)},line width=2.pt] (7.325607604296451,6.986077229941646) ellipse (4.1472177807621655cm and 1.8349814213155113cm);
\begin{scriptsize}
\draw [fill=qqqqff] (4.973357758720024,7.597518321299435) circle (2.5pt);
\draw [fill=qqqqff] (5.304255290513649,6.2451544957080865) circle (2.5pt);
\draw [fill=qqqqff] (6.416661752659802,2.906503383945089) circle (2.5pt);
\draw[color=qqqqff] (6.95,2.993726574605482) node {$f(v_{2})$};
\draw [fill=qqqqff] (5.721473917557786,6.331475590958598) circle (2.5pt);
\draw [fill=qqqqff] (8.16873076525085,2.8905754838306246) circle (2.5pt);
\draw[color=qqqqff] (8.620424033054112,2.936179177771807) node {$f(v_{3})$};
\draw [fill=qqqqff] (6.0811451477682485,6.51850463066804) circle (2.5pt);
\draw [fill=qqqqff] (9.745592876582792,2.8587196836016964) circle (2.5pt);
\draw[color=qqqqff] (10.174203747563311,2.921792328563389) node {$f(v_{4})$};
\draw [fill=qqqqff] (4.771941869802165,6.288315043333342) circle (2.5pt);
\draw [fill=qqqqff] (4.1,3.) circle (2.5pt);
\draw[color=qqqqff] (4.7,3.0368871222307376) node {$f(v_{1})$};
\draw[color=black] (4.4,7.352941884756318) node {$W_1$};
\draw [fill=qqqqff] (8.46936211636572,6.676759971960644) circle (2.5pt);
\draw [fill=qqqqff] (9.577149505413944,7.554357773674179) circle (2.5pt);
\draw [fill=qqqqff] (8.800259648159344,6.389022987792273) circle (2.5pt);
\draw [fill=qqqqff] (9.203091425995062,6.273928194124924) circle (2.5pt);
\draw [fill=qqqqff] (9.6203100530392,6.273928194124924) circle (2.5pt);
\draw[color=black] (8.3,7.396102432381575) node {$W_2$};
\draw[color=black] (6.440816377978711,8.611791190492946) node {R};
\draw[color=qqqqff] (9.65,7.820514484029924) node {$u^{(2)}$};
\draw[color=qqqqff] (5.448123782597834,7.66) node {$u^{(1)}$};
\end{scriptsize}
\end{tikzpicture}
\caption{The dashed edges of each star $S_l$ are where $E_j$ can be mapped to.  Note that all vertex sets $W_l$ are disjoint \label{embedding1}}
    	\end{minipage}  
    \end{figure}

Choose the star $S_1$ to be the image of $E_1$.
Specifically, extend the homomorphism $f$ by letting
$$f(u) = u^{(1)}, \text{ and }$$
$$f(w_i) = w_i^{(1)}, \hspace*{0.5cm} 1 \leq i \leq c.$$
Now the subhypertrees in $S$ containing $f(v_{1}),\ldots,f(v_{c})$ are connected by $f(E_1)$ in the same way they were connected originally in $T$.

Repeat the above procedure for each $E_j$.
%The vertices $v_1, \ldots, v_c$ are in $V(P)$, so they have an image in $S$ by $f$.
%There are at least $r(c)$ stars of the form
%$$S_l = \{ \{f(v_i),w_i^{(l)},u^{(l)}\},\hspace*{0.15cm} 1 \leq i \leq c\}, \hspace*{0.5cm} l=1,\ldots,s(c)$$
%in $S$ such that the sets
%$$W_l = \{w_1^{(l)},\ldots,w_c^{(l)}, u^{(l)}\}, \hspace*{0.5cm} l=1,\ldots,s(c)$$
%are pairwise disjoint and are subsets of the reservoir.
%Choose one of these stars $S_j$ such that no vertex of $W_l$ is already in the image of $f$.
Instead of mapping $E_j$ to $S_1$ each time, choose a star $S_l$ such that no vertex of $W_l$ is already in the image of $f$.
(We show below in Claim~\ref{enoughstars} that it is always possible to find such a star $S_l$.)
%Let
%$$f(u) = u^{(j)}, \text{ and }$$
%$$f(w_i) = w_i^{(j)}, \hspace*{0.5cm} 1 \leq i \leq c.$$
Extend $f$ so that it maps $E_j$ to $S_l$.
Now the subhypertrees in $S$ containing $f(v_{1}),\ldots,f(v_{c})$ are connected by $f(E_j)$ in the same way they were connected originally in $T$.

After this process has been completed for all $E_j\in \cE$, all stars of $\cE$ have an image in $S$ by $f$.
Consequently, by~(\ref{sawing6}) of Lemma~\ref{lemma:sawing}, 
all vertices of $I$ have also been embedded into $S$.
Then
$$f(T) \subset S,$$
completing the proof.

\end{proof}

\begin{claim}\label{enoughstars}
There exists some star $S_l$ in $S$ onto which to embed $E_j$, such that none of the vertices of $W_l$ has been used yet.
\end{claim}
\begin{proof}[Proof of Claim~\ref{enoughstars}]
By Lemma~\ref{lemma:stars}, there are $r(c)$ stars $S_{r}$ in $S$ onto which to embed $E_j$.
If for any $i<j$,
$$f(V(E_i)) \cap W_r \neq \emptyset,$$
then $S_r$ cannot be the image of $E_j$.
$$\vert f(V(E_i)) \cap R \vert \leq d+1,$$
so there are at most $(d+1)(j-1)$ stars $S_r$ that cannot be the image of $E_j$, because one of their vertices is already used in the image of some $E_i$.

Still, there were originally $r(c) = \frac{\epsilon^{c_j+1}m}{2(c_j^2+1)}$ stars $S_{r}$ from which to choose.
Recall that $d\geq c$, $\vert \cE \vert \geq j$, and $m \geq n$. 
Also, combining several parts of Lemma~\ref{lemma:sawing},
$$ \left( \frac{2d^2}{k}\right) n \geq \vert \cE \vert.$$
Using these facts, Inequality~\ref{eq:hier}, and~\ref{sizek},
$$r(c) = \frac{\epsilon^{c_j+1}m}{2(c_j^2+1)} \geq \frac{\epsilon^{d+1}}{2(d^2+1)}n \geq  (d+1)\frac{2d^2}{k}n > (d+1)(\vert \cE \vert-1) \geq (d+1)(j-1).$$
So there will always be at least one set $W_j$ such that for all $i<j$,
$$f(V(E_i)) \cap W_l = \emptyset.$$
Map $E_j$ onto the star $S_j$ containing $W_j$.
\end{proof}

\section{Concluding Remarks}
Note that problems of finding large matchings in Steiner systems has been extensively studied.
For some of these results, see Chapter 19 of ~\cite{CR}.
The best current bound on the size of such a matching	is due to Alon, Kim, and Spencer~\cite{AKS}, who proved a more general result implying that any Steiner triple system on $m$ vertices contains a matching of size $(m/3) - cm^{1/2}(\ln m)^{3/2}$ where $c>0$ is an absolute constant.
Likely one could extend Theorem~\ref{thm:main} along similar lines giving a better numerical bound on parameter $\mu$ as a function of $m$.

As another possible extension one could study a problem of embedding hypertrees into other designs.
In our opinion the most interesting question is whether Conjecture~\ref{conjec} is true.
In other words, does any Steiner triple system on $m$ vertices contain all hypertrees with $m-o(m)$ vertices?
Note that the main evidence for stating Conjecture~\ref{conjec} is our inability to find a counterexample.
We close with the following likely easier variant of Conjecture~\ref{conjec}
\begin{conj}\label{conjec2}
Let $d$ be a fixed constant.
Then any Steiner triple system on $n $ vertices contains all hypertrees with maximum degree $\leq d$ and $n-o(n)$ vertices.
\end{conj}

\end{document}